\documentclass[10pt,oneside,a4paper]{amsart}
\usepackage{amsmath,amsfonts,amsthm, amssymb}
\usepackage{array}
\usepackage[all,textures]{xy}
\usepackage{color}
\usepackage[orange]{xcolor}
\definecolor{blue(munsell)}{rgb}{0.0, 0.5, 0.69}
\usepackage{hyperref}
\hypersetup{hypertexnames = false, bookmarksdepth = 2, bookmarksopen = true, colorlinks, linkcolor = black, citecolor = blue(munsell), urlcolor = blue(munsell), pdfstartview={XYZ null null 1}}

\theoremstyle{plain}
\newtheorem{theorem}{Theorem}[section]
\newtheorem{proposition}[theorem]{Proposition}
\newtheorem{lemma}[theorem]{Lemma}
\newtheorem{corollary}[theorem]{Corollary}

\theoremstyle{definition}
\newtheorem{definition}[theorem]{Definition}

\theoremstyle{remark}

\newtheorem{remark}[theorem]{Remark}

\newcommand{\Kern}{\mathrm{Ker}}

\newcommand{\Beeld}{\mathrm{Im}}

\renewcommand{\lim}{\mathrm{lim}}

\newcommand{\Ext}{\mathrm{Ext}}

\newcommand{\RHom}{\mathrm{RHom}}

\newcommand{\fp}{\mathsf{fp}}

\newcommand{\op}{^{\mathrm{op}}}
\newcommand{\Ob}{\mathrm{Ob}}
\newcommand{\Z}{\mathbb{Z}}

\newcommand{\N}{\mathbb{N}}

\newcommand{\AAA}{\mathfrak{a}}
\newcommand{\BBB}{\mathfrak{b}}

\newcommand{\CCC}{\mathfrak{c}}
\newcommand{\DDD}{\mathfrak{d}}
\newcommand{\EEE}{\mathfrak{e}}

\newcommand{\Set}{\ensuremath{\mathsf{Set}} }

\newcommand{\Mod}{\ensuremath{\mathsf{Mod}} }

\newcommand{\Lex}{\ensuremath{\mathsf{Lex}} }
\newcommand{\Cont}{\ensuremath{\mathsf{Cont}}}
\newcommand{\Cocont}{\ensuremath{\mathsf{Cocont}}}

\newcommand{\Sh}{\ensuremath{\mathsf{Sh}} }

\newcommand{\Qch}{\ensuremath{\mathsf{Qch}} }

\newcommand{\Qgr}{\ensuremath{\mathsf{Qgr}} }

\newcommand{\Cat}{\ensuremath{\mathsf{Cat}} }

\newcommand{\Fun}{\ensuremath{\mathsf{Fun}}}

\newcommand{\lra}{\longrightarrow}

\newcommand{\aaa}{\ensuremath{\mathcal{A}}}
\newcommand{\bbb}{\ensuremath{\mathcal{B}}}
\newcommand{\ccc}{\ensuremath{\mathcal{C}}}
\newcommand{\ddd}{\ensuremath{\mathcal{D}}}

\newcommand{\hhh}{\ensuremath{\mathcal{H}}}

\newcommand{\LLL}{\ensuremath{\mathcal{L}}}

\newcommand{\rrr}{\ensuremath{\mathcal{R}}}

\newcommand{\ttt}{\ensuremath{\mathcal{T}}}

\newcommand{\www}{\ensuremath{\mathcal{W}}}

\CompileMatrices
\SelectTips{cm}{11}

\title{On the tensor product of linear sites and Grothendieck categories}
\author{Wendy Lowen} 
\address[Wendy Lowen]{Universiteit Antwerpen, Departement Wiskunde-Informatica, Middelheimcampus,
	Middelheimlaan 1,
	2020 Antwerp, Belgium}
\email{wendy.lowen@uantwerpen.be}

\author{Julia Ramos Gonz\'alez}
\address[Julia Ramos Gonz\'alez]{Universiteit Antwerpen, Departement Wiskunde-Informatica, Middelheimcampus,
	Middelheimlaan 1,
	2020 Antwerp, Belgium}
\email{julia.ramosgonzalez@uantwerpen.be}

\author{Boris Shoikhet}
\address[Boris Shoikhet]{Universiteit Antwerpen, Departement Wiskunde-Informatica, Middelheimcampus,
	Middelheimlaan 1,
	2020 Antwerp, Belgium}
\email{boris.shoikhet@uantwerpen.be}

\thanks{The authors acknowledge the support of the European Union for the ERC grant No 257004-HHNcdMir and the support of the Research Foundation Flanders (FWO) under Grant No. G.0112.13N}

\usepackage{pdfsync}
\begin{document}
\maketitle

\begin{abstract}
We define a tensor product of linear sites, of which we investigate the functoriality properties. Consequently we define a tensor product of Grothendieck categories based upon their representations as categories of linear sheaves. We show that our tensor product is a special case of the tensor product of locally presentable linear categories, and that the tensor product of locally coherent Grothendieck categories is locally coherent if and only if the Deligne tensor product of their abelian categories of finitely presented objects exists. We describe the tensor product of non-commutative projective schemes in terms of $\Z$-algebras, and show that for projective schemes our tensor product corresponds to the usual product scheme.
\end{abstract}

\section{Introduction}

A Grothendieck category $\ccc$ is a cocomplete abelian category with a generator and exact filtered colimits. Grothendieck categories are arguably the most important large abelian categories, second only to module categories. They play an important role in non-commutative algebraic geometry, where they are used as models for non-commutative spaces since the work of Artin, Stafford, Van den Bergh and others (\cite{artintatevandenbergh}, \cite{artinzhang2}, \cite{staffordvandenbergh}). In algebraic geometry, one of the most basic operations to be performed with schemes $X$ and $Y$ is taking their product scheme $X \times Y$. For affine schemes $\mathrm{Spec}(A)$ and $\mathrm{Spec}(B)$, this corresponds to taking the tensor product $A \otimes B$ of the underlying rings. Our aim in this paper is to define a tensor product $\ccc \boxtimes \ddd$ for arbitrary Grothendieck categories $\ccc$ and $\ddd$, such that for rings $A$ and $B$ we have
\begin{equation}
\Mod(A) \boxtimes \Mod(B) = \Mod(A \otimes B).
\end{equation}
As was originally shown in the Gabri\"el-Popescu theorem \cite{popescugabriel}, Grothendieck categories are precisely the localizations of module categories. One way of seeing this, is by describing localizations of the category $\Mod(A)$ of modules over a ring $A$ by means of data on $A$, so called Gabri\"el topologies. In the Gabri\"el-Popescu theorem, the endomorphism ring of a generator of $\ccc$ is endowed with such a Gabri\"el topology. Using the language of (linear) topologies on linear categories $\AAA$, more generally one can characterize linear functors $\AAA \lra \ccc$ which induce an equivalence $\ccc \cong \Sh(\AAA, \ttt_{\AAA}) \subseteq \Mod(\AAA)$, where $\ttt_{\AAA}$ is a certain topology on $\AAA$ and $\Sh(\AAA, \ttt_{\AAA})$ is the category of linear sheaves on $\AAA$ with respect to this topology \cite{lowenGP}. Our approach to the definition of a tensor product of Grothendieck categories consists of the following steps:
\begin{enumerate}
\item[(i)] First, we define the tensor product of linear sites $(\AAA, \ttt_{\AAA})$ en $(\BBB, \ttt_{\BBB})$ to be $(\AAA \otimes \BBB, \ttt_{\AAA} \boxtimes \ttt_{\BBB})$ for a certain tensor product topology $\ttt_{\AAA} \boxtimes \ttt_{\BBB}$ on the standard tensor product of linear categories $\AAA \otimes \BBB$.
\item[(ii)] Next, we show that the definition 
\begin{equation}\label{eqthedef}
\Sh(\AAA, \ttt_{\AAA}) \boxtimes \Sh(\BBB, \ttt_{\BBB}) = \Sh(\AAA \otimes \BBB, \ttt_{\AAA} \boxtimes \ttt_{\BBB})
\end{equation}
is a good definition for Grothendieck categories, as it is independent of the particular sites chosen in the sheaf category representations (up to equivalence of categories).
\end{enumerate}
Step (i) is carried out in \S \ref{parpartensorlin}. The topologies $\ttt_{\AAA}$ and $\ttt_{\BBB}$ naturally give rise to two ``one-sided'' topologies $\ttt_1$ and $\ttt_2$ on $\AAA \otimes \BBB$, and we put $\ttt_{\AAA} \boxtimes \ttt_{\BBB}$ equal to the supremum of $\ttt_1$ and $\ttt_2$ in the lattice of topologies on $\AAA \otimes \BBB$ (Definition \ref{deftensortop}). We further describe the corresponding operations between localizing Serre subcategories, as well as between strict localizations. In particular, we show that
\begin{equation}\label{cap}
\Sh(\AAA \otimes \BBB, \ttt_{\AAA} \boxtimes \ttt_{\BBB}) = \Sh(\AAA \otimes \BBB, \ttt_1) \cap \Sh(\AAA \otimes \BBB, \ttt_2).
\end{equation}
For compatible localizing Serre subcategories in the sense of \cite{verschoren1}, it is well known that their supremum is described by the Gabri\"el product, and using this description it is easily seen that the infimum of compatible strict localizations is simply their intersection.
However, the general case is more subtle and our analysis is based upon the construction of a semilocalizing hull (Proposition \ref{propfilt}), where a full subcategory is called \emph{semilocalizing} if it is closed under extensions and coproducts. This eventually leads to the proof of \eqref{cap} in complete generality.

An application of our constructions to the strict localizations and localizing Serre subcategories corresponding to the linear sites associated to Quillen exact categories, recovers the constructions from \cite{kaledinlowen}, which inspired the current work (\S \ref{parparex}).

Step (ii) is based upon an analysis of the functoriality of our tensor product of sites, which is carried out in \S \ref{parparfun}. An alternative approach making use of the already established tensor product of locally presentable categories going back to Kelly \cite{kelly2} \cite{kelly1} will be discussed in \S \ref{parfuture}. Since the functoriality properties established in \S \ref{parparfun} are of independent interest in the context of non-commutative geometry, we present a complete proof of step (ii) without reference to local presentablility, thus reflecting our own initial approach to the subject. A detailed discussion of the relation with the tensor product of locally presentable categories is contained in \S \ref{parlocpres}.

The classical notions of continuous and cocontinuous functors from \cite{artingrothendieckverdier1} have their linear counterparts, and we show that these types of functors are preserved by the tensor product of sites. Our main interest goes out to a special type of functors $\phi: (\AAA, \ttt_{\AAA}) \lra (\BBB, \ttt_{\BBB})$ between sites, which we call LC functors (the letters stand for ``Lemme de comparaison''). Roughly speaking, $\phi$ satisfies (LC) (Definition \ref{defLC}) if:
\begin{enumerate}
\item $\phi$ is generating with respect to $\ttt_{\BBB}$;
\item $\phi$ is fully faithful up to $\ttt_{\AAA}$;
\item $\ttt_{\AAA} = \phi^{-1}\ttt_{\BBB}$.
\end{enumerate}
The technical heart of the paper is the proof that our tensor product preserves LC functors (Proposition \ref{propLC}). Both the generating condition (1) and the fullness part of condition (2) are preserved separately. However, the faithfulness part is only preserved in combination with fullness (Lemma \ref{lemFFF}). This extends the situation for rings: surjections of rings are preserved under tensor product, injections are not (unless some flatness is assumed), but isomorphisms are obviously preserved by any functor hence also by tensoring.

The importance of LC functors $\phi: (\AAA, \ttt_{\AAA}) \lra (\BBB, \ttt_{\BBB})$ lies in the fact that they induce equivalences of categories $\Sh(\BBB, \ttt_{\BBB}) \cong \Sh(\AAA, \ttt_{\AAA})$. Further, any two representations of a given Grothendieck category $\ccc$ as $\ccc \cong \Sh(\AAA, \ttt_{\AAA})$ and $\ccc \cong \Sh(\AAA', \ttt_{\AAA'})$ can be related through a roof of LC functors. This easily yields independence of \eqref{eqthedef} from the choice of sheaf category representations (Proposition \ref{propwelldef}).

In \S \ref{parpartensorgroth}, we define the tensor product $\ccc \boxtimes \ddd$ for arbitrary Grothendieck categories $\ccc$ and $\ddd$ by formula \eqref{eqthedef} for arbitrary representations $\ccc \cong \Sh(\AAA, \ttt_{\AAA})$ and $\ddd \cong \Sh(\BBB, \ttt_{\BBB})$ (Definition \ref{defthedef}). We apply the tensor product to $\Z$-algebras and schemes. In \cite{bondalpolishchuk}, \cite{vandenbergh2}, $\Z$-algebras are used as a tool to describe non-commutative deformations of projective planes and quadrics. They are closely related to the graded algebras turning up in projective geometry, but better suited for the purpose of algebraic deformation. In particular, under some finiteness conditions, they allow nice categories of ``quasicoherent modules'' \cite{staffordvandenbergh}, \cite{polishchuk}. A (positively graded) $\Z$-algebra is a linear category $\AAA$ with $\Ob(\AAA) = \Z$ and $\AAA(n,m) = 0$ unless $n \geq m$. In \cite{dedekenlowen}, $\Z$-algebras $\AAA$ are endowed with a certain \emph{tails topology} $\ttt_{\mathrm{tails}}$ and the category $\Sh(\AAA, \ttt_{\mathrm{tails}})$ is proposed as a replacement for the category of quasicoherent modules, which exists in complete generality. We thus investigate the tensor product of two arbitrary tails sites $(\AAA, \ttt_{\AAA})$ and $(\BBB, \ttt_{\BBB})$ and show the existence of a cocontinuous functor
\begin{equation}\label{delta}
\Delta: ((\AAA \otimes \BBB)_{\Delta}, \ttt_{\mathrm{tails}}) \lra (\AAA \otimes \BBB, \ttt_{\AAA} \boxtimes \ttt_{\BBB})
\end{equation}
from the natural diagonal $\Z$-algebra $(\AAA \otimes \BBB)_{\Delta} \subseteq \AAA \otimes \BBB$ consisting of the objects $(n,n)$ for $n \in \Z$ to the tensor site (Proposition \ref{propdelta}). For a $\Z$-algebra $\AAA$, the \emph{degree} of an element in $\AAA(n,m)$ is $n - m$ and we say that $\AAA$ is generated in degree $1$ if every element can be written as a linear combination of products of elements of degree $1$ (Definition \ref{defgen}). If $\AAA$ and $\BBB$ are generated in degree $1$, then the functor $\Delta$ from \eqref{delta} is actually an LC functor (Theorem \ref{thmmaintails}). When applied to projective schemes $X$ and $Y$, by looking at the $\Z$-algebras associated to defining graded algebras which are generated in degree $1$, we obtain the following formula (Theorem \ref{thmscheme}):
\begin{equation}\label{eqscheme}
\Qch(X) \boxtimes \Qch(Y) = \Qch(X \times Y).
\end{equation}
Formula \eqref{eqscheme} is expected to hold in greater generality, at least for schemes and suitable stacks, which is work in progress.

In \S \ref{parlocpres}, we discuss the relation of our tensor product with other tensor products of categories in the literature. The existence of a tensor product of locally presentable categories goes back to \cite{kelly2}, \cite{kelly1} and features in \cite{adamekrosicky}, \cite{brandenburgchirvasitujohnsonfreyd}, \cite{cavigliahorel}, \cite{chirvasitujohnsonfreyd}. It is well known that Grothendieck categories are locally presentable. For locally $\alpha$-presentable Grothendieck categories, we use canonical sheaf representations in terms of the sites of $\alpha$-presentable objects in order to calculate our tensor product, and we show that it coincides with the tensor product as locally presentable categories. In particular, the tensor product is again locally $\alpha$-presentable. As a special case, we observe that locally finitely presentable Grothendieck categories are preserved under tensor product. In contrast, the stronger property of local coherence, which imposes the category of finitely presented objects to be abelian, is \emph{not} preserved under tensor product, as is already seen for rings. Hence, one can view the tensor product of Grothendieck categories as a solution, within the framework of abelian categories, to the non-existence, in general, of the Deligne tensor product of small abelian categories. Indeed, it was shown by L\'opez Franco in \cite{franco} that the Deligne tensor product of abelian categories $\aaa$ and $\bbb$ from \cite{deligne1} exists precisely when the finitely cocomplete tensor product $\aaa \otimes_{\mathrm{fp}} \bbb$ is abelian, and this is the case precisely when the tensor product $\Lex(\aaa) \boxtimes \Lex(\bbb)$ is locally coherent. As suggested to us by Henning Krause, we further examine the situation in terms of an $\alpha$-Deligne tensor product of $\alpha$-cocomplete abelian categories, showing that every tensor product of Grothendieck categories is accompanied by a parallel $\alpha$-Deligne tensor product of its categories of $\alpha$-presented objects for sufficiently large $\alpha$.

Our tensor product can be seen as a $k$-linear counterpart to the \emph{product} of Grothendieck toposes which is described by Johnstone in \cite{johnstone}, and its relation with the tensor product of locally presentable categories is to some extent parallel to Pitts' work in \cite{pitts}. We should note however that unlike in the case of toposes, working over $\Mod(k)$ rather than over $\Set$, the tensor product does not describe a $2$-categorical product, but instead introduces a $2$-categorical monoidal structure.  Futher, the functoriality properties we prove open up the possibility of describing a suitable monoidal $2$-category of Grothendieck categories as a $2$-localization of a monoidal $2$-category of sites at the class of LC functors. This idea applies equally well to the $\Set$-based setup. The details will appear in \cite{juliathesis}.

A combination of Pitts' approach and our description of the tensor product in terms of localizing Serre subcategories from \S \ref{partensorlocsub} leads to a natural tensor product for 
well-generated algebraic triangulated categories. The main idea is briefly sketched in \S \ref{parfuture}, its development is work in progress \cite{juliathesis}.

%
%

\vspace{0,5cm}

\noindent \emph{Acknowledgement.} The authors are very grateful to Henning Krause for suggesting the definition of an $\alpha$-version of the Deligne tensor product of abelian categories, which we have worked out in \S \ref{paralphadel} of the current version of the paper. We are also grateful for the input from an anonymous referee, who pointed out an approach discussed in \S \ref{parfuture}. We further thank Pieter Belmans for pointing out reference \cite{brandenburgchirvasitujohnsonfreyd} and Fran\c{c}ois Petit for pointing out reference \cite[\S 4.1]{lurie}.

\section{Tensor product of linear sites}\label{parpartensorlin}

Throughout, let $k$ be a commutative ground ring. For a $k$-linear category $\AAA$, we put $\Mod(\AAA) = \Fun_k(\AAA^{\op}, \Mod(k))$, the category of $k$-linear functors from $\AAA^{\op}$ to the category $\Mod(k)$ of $k$-modules.
Consider two $k$-linear categories $\AAA$ and $\BBB$, with tensor product $\AAA \otimes \BBB = \AAA \otimes_k \BBB$. The starting point for our quest for a tensor product $\boxtimes$ between Grothendieck abelian categories is the requirement that for module categories $\Mod(\AAA)$ and $\Mod(\BBB)$, we should have
\begin{equation}\label{eqmodlin}
\Mod(\AAA) \boxtimes \Mod(\BBB) = \Mod(\AAA \otimes \BBB).
\end{equation}
If we want to extend this principle to localizations of module categories, we should find a way of associating, to given localizations of $\Mod(\AAA)$ and $\Mod(\BBB)$, a new localization of $\Mod(\AAA \otimes \BBB)$. In this section, we detail three natural ways of doing this, based upon the following three isomorphic posets associated to the localization theory of $\Mod(\CCC)$ for a linear category $\CCC$ (see \S \ref{parthree}):
\begin{enumerate}
\item The poset $T$ of linear topologies on $\CCC$;
\item The poset $W$ of localizing Serre subcategories of $\Mod(\CCC)$;
\item The opposite poset $L^{\op}$ of the poset $L$ of strict localizations of $\Mod(\CCC)$.
\end{enumerate}
More precisely, taking $\CCC = \AAA \otimes \BBB$:
\begin{enumerate}
\item To topologies $\ttt_{\AAA}$ on $\AAA$ and $\ttt_{\BBB}$ on $\BBB$, we associate ``one-sided'' topologies $\ttt_1$ (induced by $\ttt_{\AAA}$) and $\ttt_2$ (induced by $\ttt_{\BBB}$) on $\AAA \otimes \BBB$, and we put $\ttt_{\AAA} \boxtimes \ttt_{\BBB} = \ttt_1 \vee \ttt_2$ in $T$ (see \S \ref{partensortop}). 
\item To localizing Serre subcategories $\www_{\AAA} \subseteq \Mod(\AAA)$ and $\www_{\BBB} \subseteq \Mod(\BBB)$, we associate the localizing Serre subcategories $\www_1, \www_2 \subseteq \Mod(\AAA \otimes \BBB)$ of objects which are in $\www_{\AAA}$ (resp. $\www_{\BBB}$) in the first (resp. second) variable, and we put $\www_{\AAA} \boxtimes \www_{\BBB} = \www_1 \vee \www_2$ in $W$ (see \S \ref{partensorlocsub}). An explicit description is based upon the construction of semilocalizing hull from \S \ref{parpreloc}.
\item To strict localizations $\LLL_{\AAA} \subseteq \Mod(\AAA)$ and $\LLL_{\BBB} \subseteq \Mod(\BBB)$, we associate the strict localizations $\LLL_1, \LLL_2 \subseteq \Mod(\AAA \otimes \BBB)$ of objects which are in $\LLL_{\AAA}$ (resp. $\LLL_{\BBB}$) in the first (resp. second) variable, and we put $\LLL_{\AAA} \boxtimes \LLL_{\BBB} = \LLL_1 \wedge \LLL_2$ in $L$ (see \S \ref{partensorstrict}). Using the relation between $W$ and $L$, one sees that actually $\LLL_{\AAA} \boxtimes \LLL_{\BBB} = \LLL_1 \cap \LLL_2$. \end{enumerate}
From the order theoretic definitions of $\ttt_{\AAA} \boxtimes \ttt_{\BBB}$, $\www_{\AAA} \boxtimes \www_{\BBB}$ and $\LLL_{\AAA} \boxtimes \LLL_{\BBB}$, we conclude that in order to establish that they correspond under the isomorphisms between $T$, $W$ and $L^{\op}$, it suffices to establish the claim for $\ttt_{\AAA}$, $\www_{\AAA}$ and $\LLL_{\AAA}$ (and similarly for $\ttt_{\BBB}$, $\www_{\BBB}$ and $\LLL_{\BBB}$). This is done in \S \ref{partensorrel}. 

An application to Quillen exact categories recovers notions from \cite{kaledinlowen} which inspired our definitions, as discussed in \S \ref{parparex}.

\subsection{Linear sites}\label{parlin}
We will use the terminology and notations from \cite[\S 2]{lowenlin}. Let $k$ be a commutative ground ring and let $\AAA$ be a small $k$-linear category. Every object $A \in \AAA$ determines a representable $\AAA$-module
$$\AAA(-,A): \AAA^{\op} \lra \Mod(k): B \longmapsto \AAA(B,A).$$
A \emph{sieve on $A$} is a submodule $R \subseteq \AAA(-,A)$. A \emph{cover system} $\rrr$ on $\AAA$ consists of specifying, for every $A \in \AAA$, a collection $\rrr(A)$ of sieves on $A$, called \emph{covering sieves on $A$} or simply \emph{covers of $A$}. 
One can list a number of properties a cover system can satisfy, as is done in \cite[\S 2.2]{lowenlin}. The most important properties are the \emph{identity axiom}, the \emph{pullback axiom}, and the \emph{glueing axiom}. If $\rrr$ satisfies the identity and pullback axioms, it is called a \emph{localizing system}. If it moreover satisfies the glueing axiom, it is called a \emph{topology}. 
Hence, what we call a topology is the $k$-linear counterpart of the notion of a Grothendieck topology.

Note that the intersection of a collection of topologies on $\AAA$ remains a topology, and $\AAA$ can be endowed with the \emph{discrete} topology for which every sieve is covering. Hence, for an arbitrary cover system $\rrr$ on $\AAA$, there exists a smallest topology $\langle \rrr \rangle_{\mathrm{top}}$ on $\AAA$ with $\rrr \subseteq \langle \rrr \rangle_{\mathrm{top}}$. If $\rrr$ is localizing, an explicit description of $\langle \rrr \rangle_{\mathrm{top}}$ is available (see \cite[\S 2.2]{lowenlin}). Consequently, the poset $T$ of topologies on $\AAA$ ordered by inclusion is a complete lattice with $\inf_i \ttt_i = \cap_{i} \ttt_i$ and $\sup_i \ttt_i = \langle \cup_i \ttt_i \rangle_{\mathrm{top}}$.

\subsection{Semilocalizing subcategories}\label{parpreloc}
Let $\ccc$ be a Grothendieck category. Recall that a \emph{localizing Serre} subcategory (\emph{localizing} subcategory for short) $\www \subseteq \ccc$ is a full subcategory closed under subquotients, extensions and coproducts. We will call a full subcategory $\www \subseteq \ccc$ \emph{semilocalizing} if it is closed under extensions and coproducts. It follows in particular that a semilocalizing subcategory $\www$ is closed under filtered colimits. As the intersection of semilocalizing (resp. localizing) subcategories is again such, for every full subcategory $\hhh \subseteq \ccc$ there is a smallest semilocalizing subcategory $\langle \hhh \rangle_{\mathrm{sloc}}$ with $\hhh \subseteq \langle \hhh \rangle_{\mathrm{sloc}}$, the \emph{semilocalizing hull} of $\hhh$, and a smallest localizing subcategory $\langle \hhh \rangle_{\mathrm{loc}}$ with $\hhh \subseteq \langle \hhh \rangle_{\mathrm{loc}}$, the \emph{localizing hull} of $\hhh$. In particular, the poset $W$ of localizing subcategories of $\ccc$ is a complete lattice with $\inf_i \www_i = \cap_i \www_i$ and $\sup_i \www_i = \langle \cup_i \www_i \rangle_{\mathrm{loc}}$.
In this section we give an explicit description of $\langle \hhh \rangle_{\mathrm{sloc}}$.

\begin{definition}
Consider $\hhh \subseteq \Ob(\ccc)$ and $C \in \ccc$. An \emph{ascending filtration} of $C$ consists of an ordinal $\alpha$ and a collection of subobjects $(M_{\beta})_{\beta \leq \alpha}$ of $C$ such that $M_0 = 0$, $i \leq j$ implies $M_i \subseteq M_j$, $M_{\beta} = \cup_{\gamma < \beta} M_{\gamma}$ if $\beta$ is a limit ordinal, and $M_{\alpha} = C$. An ascending filtration $(M_{\beta})_{\beta \leq \alpha}$ of $C$ is called an \emph{$\hhh$-filtration} provided that $M_{\beta+1}/M_{\beta} \in \hhh$ for all $\beta < \alpha$, and in this case $C$ is called \emph{$\hhh$-filtered}.
\end{definition}

\begin{proposition}\label{propfilt}
For $\hhh \subseteq \ccc$, $\langle \hhh \rangle_{\mathrm{sloc}}$ is the full subcategory of all $\hhh$-filtered objects.
\end{proposition}

\begin{proof}
Suppose first that $\hhh \subseteq \www$ for $\www$ semilocalizing. Consider an object $C \in \ccc$ with $\hhh$-filtration $(M_{\beta})_{\beta \leq \alpha}$. We show by transfinite induction that every $M_{\beta} \in \www$. The statement is true for $M_0 = 0$. Suppose $M_{\beta} \in \www$. For $M_{\beta +1}$ we have an exact sequence $0 \lra M_{\beta} \lra M_{\beta+1} \lra M_{\beta+1}/M_{\beta} \lra 0$ so since $\www$ is closed under extensions we have $M_{\beta+1} \in \www$. For a limit ordinal $\beta$, we have $M_{\beta} \in \www$ since $\www$ is closed under filtered colimits.

Next we prove that the full subcategory of $\hhh$-filtered objects is semilocalizing. Consider a coproduct $C = \oplus_{i \in I} C_i$. We may safely assume that the coproduct is indexed by successor ordinals, that is $C = \oplus_{\gamma + 1 < \alpha} C_{\gamma + 1}$ for an ordinal $\alpha$. We put $C_{\alpha} = 0$.  We inductively define an ascending filtration $(D_{\beta})_{\beta \leq \alpha}$ of $C$ with $D_0 = 0$. For a successor ordinal $\gamma + 1 \leq \alpha$, we put $D_{\gamma + 1} = D_{\gamma} \oplus C_{\gamma + 1}$ and for a limit ordinal $\beta \leq \alpha$ we put $D_{\beta} = \cup_{\gamma < \beta} D_{\gamma}$. Note that $D_{\alpha} = C$. 

By assumption, every $C_{\beta + 1}$ with $\beta + 1 < \alpha$ has an $\hhh$-filtration $(M^{\beta + 1}_{\gamma})_{\gamma \leq \alpha_{\beta + 1}}$ for some ordinal $\alpha_{\beta + 1}$. By transfinite induction on $\alpha$ we construct for every $D_{\beta}$ with $\beta \leq \alpha$ an $\hhh$-filtration refining the chosen $\hhh$-filtrations of the $D_{\gamma}$ with $\gamma < \beta$. We have the filtration $(D_0)_{0}$ for $D_0 = 0$. Suppose a $\hhh$-filtration $(P^{\beta}_{\gamma})_{\gamma \leq \theta_{\beta}}$ is chosen for $D_{\beta}$ with $\theta_{\beta}$ some ordinal. We have $D_{\beta + 1} = D_{\beta} \oplus C_{\beta + 1}$. We consider the ordinal sum $\theta_{\beta + 1} = \theta_{\beta} + \alpha_{\beta}$. We join the two $\hhh$-filtrations together into an $\hhh$-filtration $(P^{\beta + 1}_{\gamma})_{\gamma \leq \theta_{\beta + 1}}$ with $P^{\beta + 1}_{\gamma} = P^{\beta}_{\gamma}$ for $\gamma \leq \theta_{\beta}$ and $P^{\beta + 1}_{\theta_{\beta} + \gamma} = D_{\beta} \oplus M^{\beta + 1}_{\gamma}$ for $\gamma \leq \alpha_{\beta}$. For a limit ordinal $\beta \leq \alpha$, we put $\theta_{\beta} = \cup_{\gamma < \beta} \theta_{\gamma}$. We construct an $\hhh$-filtration $(P^{\beta}_{\zeta})_{\zeta \leq \theta_{\beta}}$ of $D_{\beta}$. For $\zeta < \theta_{\beta}$, there exists $\gamma < \beta$ with $\zeta < \theta_{\gamma}$, and we put $P^{\beta}_{\zeta} = P^{\gamma}_{\zeta}$. This is well defined by construction. We further put $P^{\beta}_{\theta_{\beta}} = D_{\beta}$. 

Next, consider an exact sequence 
$$\xymatrix{ 0 \ar[r] & {C'} \ar[r]_f & C \ar[r]_g & {C''} \ar[r] & 0}$$
in $\ccc$ and $\hhh$-filtrations $(M_{\beta})_{\beta \leq \alpha'}$ of $C'$ and $(N_{\beta})_{\beta \leq \alpha''}$ of $C''$. For the ordinal sum $\alpha = \alpha' + \alpha''$, we obtain an ascending filtration $(P_{\gamma})_{\gamma \leq \alpha}$ of $C$ with $P_{\gamma} = f(M_{\gamma})$ for $\gamma \leq \alpha'$ and $P_{\alpha' + \gamma} = g^{-1}(N_{\gamma})$ for $\gamma \leq \alpha''$. Note that we have $P_{\alpha'} = f(M_{\alpha'}) = g^{-1}(N_0)$ as desired since the sequence is exact. Further, we have $g^{-1}(N_{\gamma +1})/g^{-1}N_{\gamma} \cong N_{\gamma + 1}/N_{\gamma} \in \hhh$ which finishes the proof.
\end{proof}

\begin{proposition}\label{propsubquot}
If $\hhh \subseteq \ccc$ is closed under subobjects (resp. quotient objects), then the same holds for $\langle \hhh \rangle_{\mathrm{sloc}}$. In particular, if $\hhh$ is closed under subquotients, then $\langle \hhh \rangle_{\mathrm{sloc}}$ is localizing and hence $\langle \hhh \rangle_{\mathrm{sloc}} = \langle \hhh \rangle_{\mathrm{loc}}$.
\end{proposition}

\begin{proof}
Consider an exact sequence 
$$\xymatrix{ 0 \ar[r] & {C'} \ar[r]_f & C \ar[r]_g & {C''} \ar[r] & 0}$$
in $\ccc$ and an $\hhh$-filtration $(M_{\beta})_{\beta \leq \alpha}$ of $C$.
We obtain an ascending filtration $(f^{-1}(M_{\beta}))_{\beta \leq \alpha}$ of $C'$ with canonical monomorphisms $f^{-1}(M_{\beta +1})/f^{-1}(M_{\beta}) \lra M_{\beta+1}/M_{\beta}$. Hence, if $\hhh$ is closed under subobjects, this is an $\hhh$-filtration of $C'$. We obtain an ascending filtration $(g(M_{\beta}))_{\beta \leq \alpha}$ of $C''$ with canonical epimorphisms $M_{\beta +1}/M_{\beta} \lra g(M_{\beta + 1})/g(M_{\beta})$. Hence, if $\hhh$ is closed under quotient objects, this is an $\hhh$-filtration of $C''$. 
\end{proof}

\begin{corollary}\label{corWsup}
In the lattice $W$ of localizing subcategories of $\ccc$, we have $\sup_i \www_i = \langle \cup_i \www_i \rangle_{\mathrm{sloc}}$.
\end{corollary}

\begin{remark}
Recall that two (localizing) Serre subcategories $\www_1$, $\www_2$ are \emph{compatible} if $\www_1 \ast \www_2 = \www_2 \ast \www_1$ for the Gabri\"el product 
$$\www_1 \ast \www_2 = \{ C \in \ccc \,\, |\,\,  \exists W_1 \in \www_1, W_2 \in \www_2, \,\, 0 \lra W_1 \lra C \lra W_2 \lra 0 \}.$$
and in this case $\www_1 \ast \www_2$ is the smallest (localizing) Serre subcategory containing $\www_1$ and $\www_2$. 
Note that in general, by Proposition \ref{propfilt} we have $\www_1 \ast \www_2 \subseteq \langle \www_1 \cup \www_2 \rangle_{\mathrm{sloc}}$, and we have equality if and only if $\www_1$ and $\www_2$ are compatible.
\end{remark}

To end this section we describe the relation with orthogonal complements. Recall that an object $C$ is \emph{left orthogonal} to an object $D$ and $D$ is \emph{right orthogonal} to $C$ (notation $C \perp D$) provided that $\Ext^0_{\ccc}(C,D) = 0 = \Ext^1_{\ccc}(C,D)$.
For a full subcategory $\hhh \subseteq \ccc$, we obtain the following full subcategories of $\ccc$:
\begin{itemize}
\item $\hhh^{\perp} =  \{ C \in \ccc \,\, |\,\, H \perp C \,\,\, \forall H \in \hhh \};$
\item $^{\perp}\hhh =  \{ C \in \ccc \,\, |\,\, C \perp H \,\,\, \forall H \in \hhh \}$
\end{itemize}
which are called the \emph{right orthogonal complement} and the \emph{left orthogonal complement} of $\hhh$ respectively.

\begin{proposition}
For a full subcategory $\hhh \subseteq \ccc$, the left orthogonal $^{\perp} \hhh$ is semilocalizing. 
\end{proposition}

\begin{proposition}
Let $\hhh \subseteq \ccc$ be a full subcategory. We have $(\langle \hhh \rangle_{\mathrm{sloc}})^{\perp} = \hhh^{\perp}$.
\end{proposition}

\begin{proof}
Obviously $\hhh \subseteq \langle \hhh \rangle_{\mathrm{sloc}}$ implies $(\langle \hhh \rangle_{\mathrm{sloc}})^{\perp} \subseteq \hhh^{\perp}$. Since $\hhh \subseteq ^{\perp}(\hhh^{\perp})$ and $^{\perp}(\hhh^{\perp})$ is semilocalizing, we further have $\langle \hhh \rangle_{\mathrm{sloc}} \subseteq ^{\perp}(\hhh^{\perp})$ and hence $\hhh^{\perp} = (^{\perp}(\hhh^{\perp}))^{\perp} \subseteq (\langle \hhh \rangle_{\mathrm{sloc}})^{\perp}$. 
\end{proof}

\begin{corollary}\label{corperp}
For localizing subcategories $(\www_i)_i$, we have $(\sup_i \www_i)^{\perp} = \cap_i \www_i^{\perp}$. 
\end{corollary}

\subsection{Equivalent approaches to localization}\label{parthree}
Let $\AAA$ be a linear category and let $\ccc$ be a Grothendieck category. Recall that a strict localization $\LLL \subseteq \ccc$ is a full subcategory which is closed under adding isomorphic objects, for which the inclusion functor $i: \LLL \lra \ccc$ has an exact left adjoint $a: \ccc \lra \LLL$. Consider the following posets, ordered by inclusion:
\begin{enumerate}
\item The poset $T$ of linear topologies on $\AAA$;
\item The poset $W$ of localizing Serre subcategories of $\ccc$;
\item The poset $L$ of strict localizations of $\ccc$.
\end{enumerate}
It is well known that the data in (2) and (3) are equivalent, and for $\ccc = \Mod(\AAA)$ all three types of data are equivalent. Let us briefly recall the isomorphisms involved. We have an order isomorphism between $T$ and $W$, and dualities between $T$ and $L$ and between $W$ and $L$ respectively (the duality between $W$ and $L$ holds for arbitrary $\ccc$). We use the following notations. For $\ttt \in T$, $\www_{\ttt}$ and $\LLL_{\ttt}$ are the associated localizing subcategory and the associated localization. For $\www \in W$, $\ttt_{\www}$ and $\LLL_{\ttt}$ are the associated topology and the associated localization. For $\LLL \in L$, $\ttt_{\LLL}$ and $\www_{\LLL}$ are the associated topology and localizing subcategory.

We describe the involved constructions.
Consider $\ttt \in T$. A module $F \in \Mod(\AAA)$ is called a \emph{sheaf} on $\AAA$ provided that $F(A) \cong \ccc(\AAA(-,A), F) \lra \ccc(R,F)$ is an isomorphism for all $R \in \ttt(A)$. We thus obtain the full subcategory $\Sh(\AAA, \ttt)$ of sheaves on $\AAA$ and we have $\LLL_{\ttt} = \Sh(\AAA, \ttt)$. A module $F \in \Mod(\AAA)$ is called a \emph{null presheaf} if for all $x \in F(A)$ there exists $R \in \ttt(A)$ such that for all $f: A' \lra A$ in $R(A)$ we have $F(f)(x) = 0$. Then $\www_{\ttt}$ is the full subcategory of null presheaves.

Consider $\www \in W$. A subobject $R \subseteq \AAA(-,A)$ is in $\ttt_{\www}(A)$ if and only if we have $\AAA(-,A)/R \in \www$. We have $\LLL_{\www} = \www^{\perp}$.

Consider $\LLL \in L$ and let $a: \ccc \lra \LLL$ be an exact left adjoint of the inclusion $\LLL \subseteq \ccc$. We have $\www = \Kern(a) = \{ F \in \Mod(\AAA) \,\, |\,\, a(F) = 0 \}$ and a sieve $r: R \lra \AAA(-,A)$ is in $\ttt_{\LLL}$ if and only if $a(r)$ is an isomorphism.

Let us first consider the duality between $W$ and $L$ in an arbitrary Grothendieck category $\ccc$. We obtain:
\begin{proposition}\label{propinfloc}
For a collection of strict localizations $(\LLL_i)_i$ in $L$, we have $\inf_i \LLL_i = \cap_i \LLL_i$.
\end{proposition}

\begin{proof}
It suffices that $\cap_i \LLL_i$ is a strict localization, which follows from Corollary \ref{corperp} after writing $\LLL_i = \www_i^{\perp}$ for the corresponding localizing subcategories $\www_i$.
\end{proof}

Next consider the order isomorphism between $T$ and $W$ for $\ccc = \Mod(\AAA)$. Since it respects suprema, we have:

\begin{proposition}
$\www_{\sup_i \ttt_i} = \sup_i \www_{\ttt_i}$ and $\ttt_{\sup_i \www_i} = \sup_i \ttt_{\www_i}$.
\end{proposition}

\subsection{The tensor product topology}\label{partensortop}

Consider linear sites $(\AAA, \ttt_{\AAA})$ and $(\BBB, \ttt_{\BBB})$. In this section we define a topology $\ttt = \ttt_{\AAA} \boxtimes \ttt_{\BBB}$ on $\CCC = \AAA \otimes \BBB$, called the tensor product topology. For $M \in \Mod(\AAA)$ and $N \in \Mod(\BBB)$, we obtain $M \otimes N \in \Mod(\AAA \otimes \BBB)$ with $(M \otimes N)(A,B) = M(A) \otimes N(B)$.
Consider objects $A \in \AAA$, $B \in \BBB$ and covering sieves $R \in \ttt_{\AAA}(A)$ and $S \in \ttt_{\BBB}(B)$. We have $\AAA(-,A) \otimes \BBB(-,B) = \CCC(-, (A,B))$
and we  thus obtain a canonical morphism
$$\phi_{R,S}: R \otimes S \lra  \CCC(-, (A,B)).$$
We define the \emph{tensor product sieve} of $R$ and $S$ to be $$R \boxtimes S = \Beeld(\phi_{R, S}).$$
Concretely, any element in $(R \boxtimes S)(A', B')$ can be written as $\sum_{i = 1}^n \alpha_i \otimes \beta_i$ with $\alpha_i \in R(A')$ and $\beta_i \in S(B')$.
Consider the following cover systems on $\AAA \otimes \BBB$:
\begin{itemize}
\item $\rrr_{\AAA} = \{ R \boxtimes \BBB(-,B) \,\, | \,\, R \in \ttt_{\AAA}, B \in \BBB \}$;
\item $\rrr_{\BBB} = \{ \AAA(-,A) \boxtimes S \,\, |\,\, S \in \ttt_{\BBB}, A \in \AAA \}$;
\item $\rrr = \{ R \boxtimes S \,\, |\,\, R \in \ttt_{\AAA}, S \in \ttt_{\BBB} \}$.
\end{itemize}

\begin{lemma}\label{lem1}
Consider objects $A, A' \in \AAA$ and $B, B' \in \BBB$, covering sieves $R \in \ttt_{\AAA}(A)$ and $S \in \ttt_{\BBB}(B)$, and a morphism $h = \sum_{i = 1}^n f_i \otimes g_i \in \AAA(A',A) \otimes \BBB(B', B) = \CCC((A',B'), (A,B))$. We have
$$(\cap_{i = 1}^n f_i^{-1} R) \boxtimes (\cap_{i = 1}^n g_i^{-1} S) \subseteq h^{-1} (R \boxtimes S).$$
\end{lemma}

\begin{lemma}\label{lem2}
Consider objects $A \in \AAA$ and $B \in \BBB$, a covering sieve $R \in \ttt_{\AAA}(A)$, and for every morphism $h = \sum_{i = 1}^n f_i \otimes g_i \in (R \boxtimes \BBB(-,B))(A', B')$ with $A' \in \AAA$ and $B' \in \BBB$  a covering sieve $R_h \in \ttt_{\AAA}(A')$. We have
$$(R \circ (R_{f \otimes 1})_f) \boxtimes \BBB(-,B) \subseteq (R \boxtimes \BBB(-,B))\circ (R_h \boxtimes \BBB(-,B'))_h.$$
\end{lemma}

\begin{definition}\label{deftensortop}
The \emph{tensor product topology} $\ttt = \ttt_{\AAA} \boxtimes \ttt_{\BBB}$ on $\AAA \otimes \BBB$ is the smallest topology containing $\rrr_{\AAA}$ and $\rrr_{\BBB}$, that is
$$\ttt_{\AAA} \boxtimes \ttt_{\BBB} = \langle \rrr_{\AAA} \cup \rrr_{\BBB}\rangle_{\mathrm{top}}.$$
The \emph{tensor product site} of $(\AAA, \ttt_{\AAA})$ and $(\BBB, \ttt_{\BBB})$ is 
$$(\AAA, \ttt_{\AAA}) \boxtimes (\BBB, \ttt_{\BBB}) = (\AAA \otimes \BBB, \ttt_{\AAA} \boxtimes \ttt_{\BBB}) = (\CCC, \ttt).$$
\end{definition}

\begin{proposition}
\begin{enumerate}
\item[]
\item The cover systems $\ttt_1 = \rrr_{\AAA}^{\mathrm{up}}$ and  $\ttt_2 = \rrr_{\BBB}^{\mathrm{up}}$ are topologies.
\item The cover systems $\rrr^{\mathrm{up}}$ and $\ttt_1 \cup \ttt_2$ are upclosed and localizing.
\item The topoplogy $\ttt$ is the smallest topology containing $\rrr$.
\item We have $\ttt = \rrr^{\mathrm{upglue}}$ and $\ttt = (\rrr_{\AAA} \cup \rrr_{\BBB})^{\mathrm{upglue}}$.
\end{enumerate}
\end{proposition}

\begin{proof}
The cover systems $\ttt_1$, $\ttt_2$, $\ttt_1 \cup \ttt_2$ and $\rrr^{\mathrm{up}}$ are localizing by Lemma \ref{lem1} and $\ttt_1$ and $\ttt_2$ are topologies by further invoking Lemma \ref{lem2}.  It remains to prove (3). Obviously $\rrr_{\AAA} \cup \rrr_{\BBB} \subseteq \rrr$ so it remains to show that $\rrr \subseteq \ttt$. For $R \boxtimes S \in \rrr(A,B)$, we consider $R \boxtimes \BBB(-,B) \in \rrr_{\AAA}(A,B)$. For every $h = \sum_{i = 1}^n f_i \otimes g_i \in (R \boxtimes \BBB(-,B))(A', B')$, we have $\AAA(-,A') \boxtimes S \subseteq h^{-1}(R \boxtimes S)$ so $R \boxtimes S \in \ttt(A,B)$ by the glueing property.
\end{proof}

In the lattice $T$ of topologies on $\AAA \otimes \BBB$, we have 
\begin{equation}\label{eqTsup}
\ttt_{\AAA} \boxtimes \ttt_{\BBB} = \ttt_1 \vee \ttt_2.
\end{equation}

\subsection{Tensor product of (semi)localizing subcategories} \label{partensorlocsub} 
Consider linear categories $\AAA$ and $\BBB$ with $\CCC = \AAA \otimes \BBB$ and full subcategories $\www_{\AAA} \subseteq\Mod(\AAA)$ and $\www_{\BBB} \subseteq \Mod(\BBB)$.

Consider the following full subcategories of $\Mod(\CCC)$:
\begin{itemize}
\item $\www_1 = \{ F \in \Mod(\CCC) \,\, |\,\, F(-,B) \in \www_{\AAA}\,\,\, \forall B \in \BBB \}$;
\item $\www_2 = \{ F \in \Mod(\CCC) \,\, |\,\, F(A,-) \in \www_{\BBB}\,\,\, \forall A \in \AAA \}$.
\end{itemize}

\begin{proposition}
If $\www_{\AAA}$ (resp. $\www_{\BBB}$) is closed under extensions, coproducts, subobjects or quotient objects, then so is $\www_1$ (resp. $\www_2$). 
\end{proposition}

We define the \emph{tensor product} of semilocalizing subcategories $\www_{\AAA}$ and $\www_{\BBB}$ to be
$$\www_{\AAA} \boxtimes \www_{\BBB} = \langle \www_1 \cup \www_2 \rangle_{\mathrm{sloc}}.$$
The tensor product is a semilocalizing subcategory, which is localizing if $\www_{\AAA}$ and $\www_{\BBB}$ are localizing by Corollary \ref{corWsup}. More precisely, in the lattice $W$ of localizing subcategories in $\Mod(\CCC)$, we thus have 
\begin{equation}\label{eqWsup}
\www_{\AAA} \boxtimes \www_{\BBB} = \www_1 \vee \www_2.
\end{equation}

\subsection{Tensor product of strict localizations}\label{partensorstrict}
Let $\AAA$, $\BBB$ and $\CCC = \AAA \otimes \BBB$ be as before. Consider strict localizations $i_{\AAA}: \LLL_{\AAA} \lra \Mod(\AAA)$ and $i_{\BBB}: \LLL_{\BBB} \lra \Mod(\BBB)$ with exact left adjoints $a_{\AAA}: \Mod(\AAA) \lra \LLL_{\AAA}$ and $a_{\BBB}: \Mod(\BBB) \lra \LLL_{\BBB}$.
Consider the following full subcategories of $\Mod(\CCC)$:
\begin{itemize}
\item $\LLL_1 = \{ F \in \Mod(\CCC) \,\, |\,\, F(-,B) \in \LLL_{\AAA}\,\,\, \forall B \in \BBB \}$;
\item $\LLL_2 = \{ F \in \Mod(\CCC) \,\, |\,\, F(A,-) \in \LLL_{\BBB}\,\,\, \forall A \in \AAA \}$.
\end{itemize}
The natural functors
\begin{itemize}
\item $a_1: \Mod(\CCC) \lra \LLL_1: F \longmapsto (a_1(F): (A,B) \longmapsto (a_{\AAA}(F(-,B)))(A))$;
\item $a_2: \Mod(\CCC) \lra \LLL_2: F \longmapsto (a_2(F): (A,B) \longmapsto (a_{\BBB}(F(A,-)))(B))$
\end{itemize}
are readily seen to be exact left adjoints of the inclusions $i_1: \LLL_1 \lra \Mod(\CCC)$ and $i_2: \LLL_2 \lra \Mod(\CCC)$ respectively.

We define the \emph{tensor product localization} of $\LLL_{\AAA}$ and $\LLL_{\BBB}$ to be
$$\LLL_{\AAA} \boxtimes \LLL_{\BBB} = \LLL_1 \cap \LLL_2,$$
which is a strict localization by Proposition \ref{propinfloc}.
In the lattice $L$ of strict localizations of $\Mod(\CCC)$, we thus have
\begin{equation}\label{eqLinf}
\LLL_{\AAA} \boxtimes \LLL_{\BBB} = \LLL_1 \wedge \LLL_2.
\end{equation}

\subsection{Relation between the three tensor products}\label{partensorrel}
Let $\AAA$, $\BBB$ and $\CCC = \AAA \otimes \BBB$ be as before. Suppose $\AAA$ is endowed with a topology $\ttt_{\AAA}$, a localizing subcategory $\www_{\AAA}$ and a strict localization $\LLL_{\AAA}$ (with left adjoint $a_{\AAA}: \Mod(\AAA) \lra \LLL$) which correspond as in \S \ref{parthree}, an similarly $\BBB$ is endowed with a topology $\ttt_{\BBB}$, a localizing subcategory $\www_{\BBB}$ and a strict localization $\LLL_{\BBB}$ (with left adjoint $a_{\BBB}: \Mod(\BBB) \lra \LLL$) . Our aim in this section is to show that $\ttt = \ttt_{\AAA} \boxtimes \ttt_{\BBB}$, $\www = \www_{\AAA} \boxtimes \www_{\BBB}$ and $\LLL = \LLL_{\AAA} \boxtimes \LLL_{\BBB}$ (with left adjoint $a: \Mod(\CCC) \lra \LLL$) correspond as well. 

Let us first look at the relation between $\www$ and $\LLL$. We first note the following:

\begin{proposition}\label{propWL}
The localizing subcategory $\www_1$ (resp. $\www_2$) and the strict localization $\LLL_1$ (resp. $\LLL_2$) correspond under the isomorphism between $W$ and $L$.
\end{proposition}

\begin{proof}
Consider $\www_{\AAA} = \Kern(a_{\AAA})$ and $\www_{\BBB} = \Kern(a_{\BBB})$ and let $\www_1$, $\www_2$ be as defined in \S \ref{partensorlocsub}. By direct inspection, we have $\Kern(a_1) = \www_1$ and $\Kern(a_2) = \www_2$.
\end{proof}

\begin{corollary}\label{corcor}
The localizing subcategory $\www_{\AAA} \boxtimes \www_{\BBB}$ and the strict localization $\LLL_{\AAA} \boxtimes \LLL_{\BBB}$ correspond under the isomorphism between $W$ and $L$.
\end{corollary}

\begin{proof}
This follows from Proposition \ref{propWL} and equations \eqref{eqWsup} and \eqref{eqLinf}.
\end{proof}

Next we look at the relation between $\ttt$ and $\www$. Again, we must first establish the relation between $\ttt_1$ and $\www_1$ and between $\ttt_2$ and $\www_2$.

\begin{proposition}\label{propTW}
The topology $\ttt_1$ and the localizing subcategory $\www_1$ (resp. the topology $\ttt_2$ and the localizing subcategory $\www_2$) correspond under the isomorphism between $T$ and $L$.
\end{proposition}

\begin{proof}
It suffices to show the following inclusions: (1) $\ttt_1 \subseteq \ttt_{\www_1}$, (2) $\www_1 \subseteq \www_{\ttt_1}$.

For (1) consider, for $R \in \ttt_{\AAA}(A)$ and $B \in \BBB$, the sieve $R \boxtimes \BBB(-,B)$. The exact sequence $R \lra \AAA(-,A) \lra \AAA(-,A)/R \lra 0$ gives rise to the exact sequence
$R \otimes \BBB(-,B) \lra \CCC(-,(A,B)) \lra (\AAA(-,A)/R) \otimes \BBB(-,B) \lra 0$ and hence $F = \CCC(-, (A,B))/(R \boxtimes \BBB(-,B)) \cong (\AAA(-,A)/R) \otimes \BBB(-,B)$. It suffices to show that $F \in \www_1$. Now we have $Z = \AAA(-,A)/R \in \www_{\AAA}$ and for $B' \in \BBB$,  $F(-,B') = Z \otimes \BBB(B',B)$ is a colimit of objects in $\www_{\AAA}$, hence it is itself in $\www_{\AAA}$ as desired.

For (2), consider $F \in \www_1$ and $x \in F(A,B) = F(-,B)(A)$. Since $F(-,B) \in \www_{\AAA}$, there exists $R \in \ttt_{\AAA}$ such that for all $f \in R(A')$ we have $F(f \otimes 1): F(A, B) \lra F(A', B): x \longmapsto 0$. It now suffices to consider $R \boxtimes \BBB(-,B) \in \ttt_1(A,B)$. For every element $h = \sum_i f_i \otimes g_i = \sum_i(1 \otimes g_i)(f_i \otimes 1) \in (R \boxtimes \BBB(-,B))(A', B')$, we have $F(h)(x) = 0$ as desired.
\end{proof}

\begin{corollary}
The topology $\ttt_{\AAA} \boxtimes \ttt_{\BBB}$ and the localizing subcategory $\www_{\AAA} \boxtimes \www_{\BBB}$ correspond under the isomorphism between $T$ and $W$.
\end{corollary}

\begin{proof}
This follows from Proposition \ref{propTW} and equations \eqref{eqTsup} and \eqref{eqWsup}.
\end{proof}

\subsection{Exact categories}\label{parparex}

The following setup from \cite{kaledinlowen} inspired our definitions. Let $\AAA$ be an exact category in the sense of Quillen, and let $\Lex(\AAA) \subseteq \Mod(\AAA)$ be the full subcategory of left exact functors, that is additive functors $F: \AAA^{\op} \lra \Mod(k)$ which send conflations $0 \lra K \lra D \lra C \lra 0$ to short exact sequences $0 \lra F(C) \lra F(D) \lra F(K)$. The category $\AAA$ can be endowed with the single deflation topology, for which a sieve is covering if and only if it contains a deflation for the exact structure. The category of sheaves for this topology is precisely $\Lex(\AAA)$, and the corresponding localizing Serre subcategory $\www_{\AAA}$ is the category of weakly effaceable modules, that is, modules $M \in \Mod(\AAA)$ such that for every $x \in M(A)$ there exists a deflation $A' \lra A$ for which $M(A) \lra M(A')$ maps $x$ to zero. 

Now let $\AAA$ and $\BBB$ be exact categories. In \cite[\S 2.6]{kaledinlowen}, the full category $\mathsf{Lex}(\BBB, \AAA^{\op}) \subseteq \Mod(\BBB \otimes \AAA^{\op})$ of bimodules $M$ for which every $M(-,A) \in \Lex(\BBB)$ and every $M(B,-) \in \Lex(\AAA^{\op})$ for the natural ``dual'' exact structure on $\AAA^{\op}$ is introduced. With the definition from \S \ref{partensorstrict}, we thus have 
$$\mathsf{Lex}(\BBB, \AAA^{\op}) = \Lex(\BBB) \boxtimes \Lex(\AAA^{\op})$$
and in \cite[Prop. 2.22]{kaledinlowen}, the relation with the localizing Serre subcategory $\www_{\BBB} \boxtimes \www_{\AAA^{\op}}$ from Corollary \ref{corcor} was demonstrated using the description of the Gabri\"el product.

In \cite{kaledinlowen}, it is argued that $\mathsf{Lex}(\BBB, \AAA^{\op})$ is the correct bimodule category to consider between exact categories, where we look at bimodules contravariant in $B \in \BBB$ and covariant in $A \in \AAA$. In particular, it is shown that over a field $k$, Hochschild cohomology of $\AAA$ in the sense of \cite{keller6}, and of $\Lex(\AAA)$ in the sense of \cite{lowenvandenberghhoch} is equal to $$HH^n(\AAA) = \Ext^n_{\Lex(\AAA, \AAA^{\op})}(1_{\AAA}, 1_{\AAA}).$$
It is not clear how this approach could be extended to more general sites, as it makes essential use of the existence of a natural ``dual site'' for the site associated to an exact category.

\section{Functoriality of the tensor product of linear sites}\label{parparfun}
Let $\AAA$ and $\BBB$ be $k$-linear categories as before. Let us return to the starting point for our quest for a tensor product $\boxtimes$ between Grothendieck abelian categories, namely the requirement that
\begin{equation}\label{eqmodlin2}
\Mod(\AAA) \boxtimes \Mod(\BBB) = \Mod(\AAA \otimes \BBB).
\end{equation}
Using the $2$-categorical structure of the category $\Cat(k)$ of $k$-linear categories, functors and natural transformations, it is not hard to see that $\boxtimes$ can be defined based upon \eqref{eqmodlin2} for module categories $\ccc$. A module category $\ccc$ is intrinsically characterized by the existence of a set of  finitely generated projective generators, and different choices of generators give rise to Morita equivalent linear categories. For Morita bimodules $M$ between $\AAA$ and $\AAA'$ and $N$ between $\BBB$ and $\BBB'$, it is readily seen that $M \otimes N$ with $M \otimes N((A,B), (A', B')) = M(A,A') \otimes N(B,B')$ defines a Morita bimodule between $\AAA \otimes \BBB$ and $\AAA' \otimes \BBB'$. Our aim in this section is to develop the necessary tools in order to extend this situation from module categories to arbitrary Grothendieck categories. Rather than focussing on bimodules, we first focus on functors between sites. 
The underlying idea is that any equivalence between sheaf categories can be represented by a roof of LC functors between sites, where an LC functor is a particular kind of functor which induces an equivalence between sheaf categories. Roughly speaking, an LC functor $\phi: (\AAA, \ttt_{\AAA}) \lra (\CCC, \ttt_{\CCC})$ is generating with respect to $\ttt_{\CCC}$, fully faithfull up to $\ttt_{\AAA}$, and has $\phi^{-1} \ttt_{\CCC} = \ttt_{\AAA}$ (Definition \ref{defLC}). The main result of this section is that LC functors are preserved under tensor product of sites (Proposition \ref{propLC}).

\subsection{Functors}\label{parfun}

Consider linear categories $\AAA$ and $\CCC$ and a linear functor $\phi: \AAA \lra \CCC$. Suppose $\ttt_{\AAA}$ and $\ttt_{\CCC}$ are cover systems on the respective categories.

\begin{definition}\label{defct}
Suppose $\ttt_{\AAA}$ and $\ttt_{\CCC}$ are localizing. The functor $\phi: (\AAA, \ttt_{\AAA}) \lra (\CCC, \ttt_{\CCC})$ is called \emph{continuous} provided that $\phi^{\ast}: \Mod(\CCC) \lra \Mod(\AAA)$ preserves sheaves, and hence restricts to a functor $\phi_s: \Sh(\CCC, \ttt_{\CCC}) \lra \Sh(\AAA, \ttt_{\AAA})$. 
\end{definition}

\begin{definition}\label{defcoverct}\cite[Def. 2.11]{lowenlin}
The functor $\phi: (\AAA, \ttt_{\AAA}) \lra (\CCC, \ttt_{\CCC})$ is called \emph{cocontinuous} provided that for every $A \in \AAA$ and $R \in \ttt_{\CCC}(\phi(A))$ there exists $S \in \ttt_{\AAA}(A)$ with $\phi S \subseteq R$.
\end{definition}

\begin{remark}
Suppose $\ttt_{\AAA}$ and $\ttt_{\CCC}$ are topologies. Continuous morphisms are the linear counterpart of the continuous morphisms from \cite{artingrothendieckverdier1}, and cocontinuous morphisms are the linear counterpart of the cocontinuous morphisms from \cite{artingrothendieckverdier1}. In \cite{lowenlin}, the term ``cover continuous'' is used for what we call cocontinuous here.
\end{remark}

Next we recall some special conditions (see \cite[\S 2.5]{lowenlin}).

\begin{definition}\label{defLC}
Consider a linear functor $\phi: \AAA \lra \CCC$. 
\begin{enumerate}
\item Suppose $\CCC$ is endowed with a cover system $\ttt_{\CCC}$. We say that $\phi: \AAA \lra (\CCC, \ttt_{\CCC})$ satisfies 
\begin{itemize}
\item[(G)] if for every $C \in \CCC$ there is a covering family $(\phi(A_i) \lra C)_i$ for $\ttt_{\CCC}$.
\end{itemize}
\item Suppose $\AAA$ is endowed with a cover system $\ttt_{\AAA}$. We say that $\phi: (\AAA, \ttt_{\AAA}) \lra \CCC$ satisfies
\begin{itemize}
\item[(F)] if for every $c: \phi(A) \lra \phi(A')$ in $\CCC$ there exists a covering family $a_i: A_i \lra A$ for $\ttt_{\AAA}$ and $f_i: A_i \lra A'$ with $c\phi(a_i) = \phi({f}_i)$;
\item[(FF)] if for every $a: A \lra A'$ in $\AAA$ with $\phi(a) = 0$ there exists a covering family $a_i: A_i \lra A$ for $\ttt_{\AAA}$ with $aa_i = 0$. 
\end{itemize}
\item Suppose $\AAA$ and $\CCC$ are endowed with cover systems $\ttt_{\AAA}$ and $\ttt_{\CCC}$ respectively. We say that $\phi: (\AAA, \ttt_{\AAA}) \lra (\CCC, \ttt_{\CCC})$ satisfies
\begin{itemize}
\item[(LC)] if $\phi$ satisfies (G) with respect to $\ttt_{\CCC}$, (F) and (FF) with respect to $\ttt_{\AAA}$, and we further have $\ttt_{\AAA} = \phi^{-1} \ttt_{\CCC}$.
\end{itemize}
\end{enumerate}
\end{definition}

We have the following ``Lemme de comparaison'' (see \cite{artingrothendieckverdier1}, \cite{lowenGP}, \cite{lowenlin}):
\begin{theorem}
If $\phi: (\AAA, \ttt_{\AAA}) \lra (\CCC, \ttt_{\CCC})$ satisfies (LC) and $\ttt_{\CCC}$ is a topology, then $\ttt_{\AAA}$ is a topology, $\phi$ is continuous and $\phi_s: \Sh(\CCC, \ttt_{\CCC}) \lra \Sh(\AAA, \ttt_{\AAA})$ is an equivalence of categories. 
\end{theorem}

The following lemma, which is easily proven by induction, will be used later on:

\begin{lemma}\label{lemcoverind}
Suppose the functor $\phi: (\AAA, \ttt_{\AAA}) \lra (\CCC, \ttt_{\CCC})$ satisfies (F) and $\ttt_{\AAA}$ is a topology and consider morphisms $c_i: \phi(A) \lra \phi(A')$ for $i = 1, \dots, n$. There exists a collection of morphisms $h_j: A_i \lra A$ for $j \in J$ with $\langle h_j \rangle \in \ttt_{\AAA}$ and $g_{ij}: A_i \lra A'$ such that $c_i \phi(h_j) = \phi(g_{ij})$ for all $i = 1, \dots, n$ and $j \in J$.
\end{lemma}

\subsection{Tensor product of functors}\label{partensorfun}

Consider linear categories $\AAA$, $\BBB$, $\CCC$ and $\DDD$ and
linear functors $\phi: \AAA \lra \CCC$ and $\psi: \BBB \lra \DDD$. 
Consider the tensor product functor
$$\phi \otimes \psi: \AAA \otimes \BBB \lra \CCC \otimes \DDD.$$
We have $\phi \otimes \psi = (1 \otimes \psi)(\phi \otimes 1)$ for $$\phi \otimes 1: \AAA \otimes \BBB \lra \CCC \otimes \BBB$$ and $$1 \otimes \psi: \CCC \otimes \BBB \lra \CCC \otimes \DDD.$$ 
Suppose $\ttt_{\AAA}$, $\ttt_{\BBB}$, $\ttt_{\CCC}$ and $\ttt_{\DDD}$ are cover systems on the respective categories.

\begin{proposition}
Suppose all cover systems are localizing. If $\phi$ and $\psi$ are continuous, then so is $\phi \otimes \psi$.
\end{proposition}

\begin{proof}
We have to look at the functor $(\phi \otimes \psi)^{\ast}: \Mod(\CCC \otimes \DDD) \lra \Mod(\AAA \otimes \BBB)$. According to \S \ref{partensorstrict}, \S \ref{partensorrel} we have $F \in \Sh(\CCC \otimes \DDD, \ttt_{\CCC} \boxtimes \ttt_{\DDD})$ if and only if $F(-,-)$ is a sheaf in both variables for $\ttt_{\CCC}$ and $\ttt_{\DDD}$ respectively. It readily follows that $(\phi \otimes \psi)^{\ast}F = F(\phi(-), \psi(-))$ is a sheaf in the first variable for $\ttt_{\AAA}$ as soon as $\phi$ is continuous, and a sheaf in the second variable for $\ttt_{\BBB}$ as soon as $\psi$ is continuous.
\end{proof}

The following is easy to check:

\begin{lemma}\label{lemcalc}
Consider $A \in \AAA$, $B \in \BBB$, and sieves $R \subseteq \AAA(-,A)$ and $S \subseteq \BBB(-,B)$.  As sieves on $(\phi(A), B)$, we have
$$
\phi(R) \boxtimes S = (\phi \otimes 1)(R \boxtimes S).
$$
\end{lemma}

\begin{proposition}\label{propcovertimes}
If $\phi$ and $\psi$ are cocontinuous, then so is $\phi \otimes \psi$.
\end{proposition}

\begin{proof}
Since cocontinuous functors are stable under composition, it suffices to consider $\psi = 1: (\BBB, \ttt_{\BBB}) \lra (\BBB, \ttt_{\BBB})$. 
We are to show that $\phi \otimes 1: (\AAA \otimes \BBB, \ttt_{\AAA} \boxtimes \ttt_{\BBB}) \lra (\CCC\otimes \BBB, \ttt_{\CCC} \boxtimes \ttt_{\BBB})$ is cocontinuous. By \cite[Lem. 2.12]{lowenlin}, it suffices to show that $\phi \otimes 1$ is cocontinuous with respect to the localizing cover systems $\LLL_{\AAA, \BBB} = \{ R \boxtimes S \,\, |\,\, R \in \ttt_{\AAA}, S \in \ttt_{\BBB} \}^{\mathrm{up}}$ and $\LLL_{\CCC, \BBB} = \{ T \boxtimes S \,\, |\,\, T \in \ttt_{\CCC}, S \in \ttt_{\BBB} \}^{\mathrm{up}}$. Thus, consider $T \boxtimes S$ with $T \in \ttt_{\CCC}(\phi(A)), S \in \ttt_{\BBB}(B)$. By the assumption there exists $R \in \ttt_{\AAA}$ with $\phi R \subseteq T$. Consequently, by Lemma \ref{lemcalc} we have $(\phi \otimes 1)(R \boxtimes S) \subseteq T \boxtimes S$ as desired.
\end{proof}

\begin{lemma}\label{lemG}
Suppose the functor $\phi$ satisfies (G) with respect to $\ttt_{\CCC}$. Then the functor $\phi \otimes 1$ satisfies (G) with respect to $\ttt_{\CCC} \boxtimes \ttt_{\BBB}$.
\end{lemma}

\begin{proof}
Consider $(C, B) \in \CCC \otimes \BBB$. There exists $R = \langle f_i: \phi(A_i) \lra C \rangle \in \ttt_{\CCC}(C)$. It is easily seen that $R \boxtimes \BBB(-,B) = \langle f_i \otimes 1_B: (\phi \otimes 1)(A_i, B) \lra (C,B) \rangle$.
\end{proof}

\begin{lemma}\label{lemF}
Suppose the functor $\phi$ satisfies (F) with respect to $\ttt_{\AAA}$. Then the functor $\phi \otimes 1$ satisfies (F) with respect to $\ttt_{\AAA} \boxtimes \ttt_{\BBB}$.
\end{lemma}

\begin{proof}
Consider a morphism $h = \sum_{i = 1}^n c_i \otimes b_i: (\phi(A), B) \lra (\phi(A'), B')$. We proceed by induction. Suppose we have a collection $(a_{\alpha}: A_{\alpha} \lra A)_{\alpha}$ of morphisms in $\AAA$  with $\langle a_{\alpha} \rangle \in \ttt_{\AAA}(A)$ (and hence $\langle a_{\alpha} \otimes 1_B \rangle \in \ttt_{\AAA} \boxtimes \ttt_{\BBB}(A, B)$) and for $i \in \{ 1, \dots, m-1 \}$ with $m \leq n$ and $\alpha$, there exists $g^i_{\alpha}: (A_{\alpha}, B) \lra (A',B)$ with $(\phi \otimes 1)(g^i_{\alpha}) = (c_i \otimes b_i)(\phi(a_{\alpha}) \otimes 1_B)$. We show that the same holds for $i = m$. To this end, we consider $(c_m \otimes b_m)(\phi(a_{\alpha}) \otimes 1_B) = c_m \phi(a_{\alpha}) \otimes b_m$. For $c_m \phi(a_{\alpha}): \phi(A_{\alpha}) \lra \phi(A')$, since $\phi$ satisfies (F) there is a collection $(a^{\alpha}_{\beta}: A^{\alpha}_{\beta} \lra A_{\alpha})_{\beta}$ with $\langle a^{\alpha}_{\beta} \rangle \in \ttt_{\AAA}(A_{\alpha})$ and there exist morphisms $f^{\alpha}_{\beta}: A^{\alpha}_{\beta} \lra A'$ with $\phi(f^{\alpha}_{\beta}) = c_m \phi(a_{\alpha})\phi(a^{\alpha}_{\beta})$. Consequently, the collection of compositions $a_{\alpha} a^{\alpha}_{\beta}: A^{\alpha}_{\beta} \lra A$ are such that $\langle a_{\alpha} a^{\alpha}_{\beta} \rangle \in \ttt_{\AAA}(A)$ by the glueing property (and hence 
$\langle a_{\alpha} a^{\alpha}_{\beta} \otimes 1_B \rangle \in \ttt_{\AAA} \boxtimes \ttt_{\BBB}(A, B)$). For $i \in \{ 1, \dots, m-1 \}$, we have $(\phi \otimes 1)(g^i_{\alpha}(a^{\alpha}_{\beta} \otimes 1_B)) = (c_i \otimes b_i)(\phi(a_{\alpha}a^{\alpha}_{\beta}) \otimes 1_B)$ by the induction hypothesis. For $m$, we have $(\phi \otimes 1)(f^{\alpha}_{\beta} \otimes b_m) = (c_m \otimes b_m)(\phi(a_{\alpha}a^{\alpha}_{\beta}) \otimes 1_B)$ as desired.
\end{proof}

\begin{lemma}\label{lemFFF}
Suppose the functor $\phi$ satisfies (F) and (FF) with respect to $\ttt_{\AAA}$. Then the functor $\phi \otimes 1$ satisfies (F) and (FF) with respect to $\ttt_{\AAA} \boxtimes \ttt_{\BBB}$.
\end{lemma}

\begin{proof}
Consider a morphism $h = \sum_{i =1}^n a_i \otimes b_i: (A,B) \lra (A', B')$ such that $0 = (\phi \otimes 1)(h) = \sum_{i = 1}^n \phi(a_i) \otimes b_i$. Let $(c_\lambda)_{\lambda \in \Lambda}$ be a collection of generators of the $k$-module $\CCC(\phi(A), \phi(A'))$ such that $\{ 1, \dots, n \} \subseteq \Lambda$ and $c_i = \phi(a_i)$ for $i \in \{ 1, \dots, n \}$. Put $b_{\lambda} = 0$ for $\lambda \in \Lambda \setminus \{1, \dots, n\}$. We thus have $0 = \sum_{\lambda \in \Lambda} c_{\lambda} \otimes b_{\lambda} \in \CCC(\phi(A), \phi(A')) \otimes \BBB(B, B')$. According to \cite[Lem. 6.4]{eisenbudcommag}, there exist $\bar{b}_j \in \BBB(B, B')$ for $j \in J$ and $\kappa_{\lambda, j} \in k$ such that $(\kappa_{\lambda, j})_{\lambda \in \Lambda, j \in J}$ contains only finitely many non-zero elements, such that the following hold:
\begin{enumerate}
\item $b_{\lambda} = \sum_{j} \kappa_{\lambda, j} \bar{b}_j$ for all $\lambda \in \Lambda$;
\item $0 = \sum_{\lambda} \kappa_{\lambda, j} c_{\lambda} = \sum_{i = 1}^n \kappa_{i,j} \phi(a_i) + \sum_{\lambda \in \Lambda \setminus \{1, \dots, n\}} \kappa_{\lambda, j} c_{\lambda}$ for all $j \in J$.
\end{enumerate}
Using (F) for $\phi$, we will first realize the right hand side of (2) as being in the image of $\phi$ up to a covering. Let $\Lambda_0 \subseteq \Lambda$ contain those $\lambda$'s for which there exists $j \in J$ with $\kappa_{\lambda, j} \neq 0$. Hence $\Lambda_0$ is finite.  By Lemma \ref{lemcoverind}, there exists a collection $h_{\sigma}: A_{\sigma} \lra A$ for $\sigma \in \Sigma$ with $\langle h_{\sigma} \rangle \in \ttt_{\AAA}(A)$ and $g_{\lambda, \sigma}: A_{\sigma} \lra A'$ such that $c_{\lambda} \phi(h_{\sigma}) = \phi(g_{\lambda, \sigma})$ for $\lambda \in \Lambda_0$ and $\sigma \in \Sigma$. Further, we may clearly suppose that
\begin{equation}\label{eqg}
 g_{i, \sigma} = a_i h_{\sigma}
 \end{equation}
 for $i = 1, \dots, n$.
Hence, for $j \in J$ and $\sigma \in \Sigma$, from (2) we obtain:
$$0 = \sum_{\lambda} \kappa_{\lambda, j} c_{\lambda} \phi(h_{\sigma})  = \phi(\sum_{\lambda} \kappa_{\lambda, j} g_{\lambda, \sigma}).$$
Using (FF) for $\phi$, for every $\sigma \in \Sigma$ we obtain a collection $h^{\sigma}_{\omega}: A^{\sigma}_{\omega} \lra A_{\sigma}$ for $\omega \in \Omega_{\sigma}$ with $\langle h^{\sigma}_{\omega} \rangle \in \ttt_{\AAA}(A_{\sigma})$ such that for every $\omega \in \Omega_{\sigma}$
\begin{equation}\label{eqzero}
0 = \sum_{\lambda} \kappa_{\lambda, j} g_{\lambda, \sigma} h^{\sigma}_{\omega}.
\end{equation}
Now consider the collection $h_{\sigma} h^{\sigma}_{\omega}: A^{\sigma}_{\omega} \lra A$ for $\sigma \in \Sigma$ and $\omega \in \Omega_{\sigma}$. By the glueing property we have $\langle h_{\sigma} h^{\sigma}_{\omega} \rangle \in \ttt_{\AAA}(A)$. Further, we have $\langle h_{\sigma} h^{\sigma}_{\omega} \otimes 1_B \rangle \in \ttt_{\AAA} \boxtimes \ttt_{\BBB}(A,B)$.
We claim that $h$ becomes zero on this covering sieve of $(A,B)$. We have $h = \sum_{i=1}^n a_i \otimes b_i = \sum_j (\sum_{i =1}^n \kappa_{i, j} a_i) \otimes \bar{b}_j$. We compute
$$
h (h_{\sigma} h^{\sigma}_{\omega} \otimes 1) = \sum_{j \in J}(\sum_{i = 1}^n \kappa_{i, j} a_i h_{\sigma} h^{\sigma}_{\omega} \otimes \bar{b}_j).
$$
Consider the expressions
$$x = \sum_{j \in J}(\sum_{\lambda \in \Lambda} \kappa_{\lambda, j} g_{\lambda, \sigma} h^{\sigma}_{\omega} \otimes \bar{b}_j)$$
and
$$y = \sum_{j \in J}(\sum_{\lambda \in \Lambda \setminus \{1, \dots, n\}} \kappa_{\lambda, j} g_{\lambda, \sigma} h^{\sigma}_{\omega} \otimes \bar{b}_j)
= \sum_{\lambda \in \Lambda \setminus \{1, \dots, n\}} (g_{\lambda, \sigma} h^{\sigma}_{\omega} \otimes \sum_{j \in J} \kappa_{\lambda, j} \bar{b}_j).$$
Using equation \eqref{eqg}, we clearly have $$x = h (h_{\sigma} h^{\sigma}_{\omega} \otimes 1) + y.$$
By equation \eqref{eqzero}, we have $x = 0$.
By definition and by condition (2) above, for $\lambda \in \Lambda \setminus \{1, \dots, n\}$, we have $0 = b_{\lambda} = \sum_{j \in J} \kappa_{\lambda, j} \bar{b}_j$ so also $y = 0$. We conclude that $h (h_{\sigma} h^{\sigma}_{\omega} \otimes 1) = 0$ as desired.
\end{proof}

\begin{lemma}
If the functor $\phi: (\AAA, \ttt_{\AAA}) \lra (\CCC, \ttt_{\CCC})$ satisfies (LC), then so does the functor $\phi \otimes 1: (\AAA \otimes \BBB, \ttt_{\AAA} \boxtimes \ttt_{\BBB}) \lra (\CCC\otimes \BBB, \ttt_{\CCC} \boxtimes \ttt_{\BBB})$.
\end{lemma}

\begin{proof}
By Lemmas \ref{lemG}, \ref{lemF}, and \ref{lemFFF}, $\phi \otimes 1$ satisfies (G), (F) and (FF). We have $\ttt_{\AAA} = \phi^{-1} \ttt_{\CCC}$, and it remains to show that $$\ttt_{\AAA} \boxtimes \ttt_{\BBB} = (\phi \otimes 1)^{-1} (\ttt_{\CCC} \boxtimes \ttt_{\BBB}).$$
To prove the inclusion $\ttt_{\AAA} \boxtimes \ttt_{\BBB} \subseteq (\phi \otimes 1)^{-1} (\ttt_{\CCC} \boxtimes \ttt_{\BBB})$, it suffices to look at a sieve $R \boxtimes S$ with $\phi(R) \in \ttt_{\CCC}$ and $S \in \ttt_{\BBB}$. It immediately follows from Lemma \ref{lemcalc} that $(\phi \otimes 1)(R \boxtimes S) \in \ttt_{\CCC} \boxtimes \ttt_{\BBB}$. 

For the other inclusion, by \cite[Prop. 2.16]{lowenlin}, it suffices to show that $\phi \otimes 1: (\AAA \otimes \BBB, \ttt_{\AAA} \boxtimes \ttt_{\BBB}) \lra (\CCC\otimes \BBB, \ttt_{\CCC} \boxtimes \ttt_{\BBB})$ is cocontinuous. By \cite[Lem. 2.15]{lowenlin}, $\phi$ is cocontinuous, whence it follows by Proposition \ref{propcovertimes} that $\phi \otimes 1$ is cocontinuous as desired.
\end{proof}

\begin{proposition}\label{propLC}
If the functors $\phi: (\AAA, \ttt_{\AAA}) \lra (\CCC, \ttt_{\CCC})$ and $\psi: (\BBB, \ttt_{\BBB}) \lra (\DDD, \ttt_{\DDD})$ both satisfy (G) (resp. (F), resp. (F) and (FF), resp. (LC)) , then so does the functor $\phi \otimes \psi: (\AAA \otimes \BBB, \ttt_{\AAA} \boxtimes \ttt_{\BBB}) \lra (\CCC \otimes \DDD, \ttt_{\CCC} \boxtimes \ttt_{\DDD})$. 
\end{proposition}

\section{Tensor product of Grothendieck categories}\label{parpartensorgroth}

Based upon the previous sections, in \S \ref{partensorgroth} we are finally in a position to define the tensor product of Grothendieck categories $\ccc \cong \Sh(\AAA, \ttt_{\AAA})$ and $\ddd \cong \Sh(\BBB, \ttt_{\BBB})$ to be given by $$\ccc \boxtimes \ddd = \Sh(\AAA \otimes \BBB, \ttt_{\AAA} \boxtimes \ttt_{\BBB}).$$ Functoriality of the tensor product of sites ensures that $\ccc \boxtimes \ddd$ is welldefined up to equivalence of categories.  For an alternative approach ensuring welldefinedness by making use of the already established tensor product of locally presentable categories, we refer to \S \ref{parfuture}. 

The remainder of this section is devoted to an application of our tensor product to $\Z$-algebras and schemes. In \S \ref{parzalg} we provide a nice realisation of the tensor product of $\Z$-algebras, while in \S \ref{parqcohproj}, we show that for projective schemes $X$ and $Y$ we have
\begin{equation}\label{eqschin}
\Qch(X) \boxtimes \Qch(Y) = \Qch(X \times Y).
\end{equation}
This result generalizes to non-commutative projective schemes, and our proof is actually based upon the results in \S \ref{parzalg}. Here, we use $\Z$-algebras as models for non-commutative schemes following \cite{bondalpolishchuk}, \cite{vandenbergh2}, \cite{staffordvandenbergh}, \cite{dedekenlowen}, and to a $\Z$-algebra $\AAA$ we can associate a certain category $\Qch(\AAA)$ which replaces the quasicoherent sheaves, and which is obtained as a linear sheaf category with respect to a certain topology. For two $\Z$-algebras $\AAA$ and $\BBB$ generated in degree $1$, there is a naturally associated diagonal $\Z$-algebra $(\AAA \otimes \BBB)_{\Delta}$, for which we show that
\begin{equation}\label{eqzalgin}
\Qch(\AAA) \boxtimes \Qch(\BBB) \cong \Qch((\AAA \otimes \BBB)_{\Delta}).
\end{equation}
The relation between \eqref{eqschin} and \eqref{eqzalgin} is provided by graded algebras (generated in degree $1$), which on the one hand are used to represent schemes through the Proj construction, and which on the other hand give rise to associated $\Z$-algebras.

\subsection{Tensor product of Grothendieck categories}\label{partensorgroth}
Let $\ccc$ be a $k$-linear Grothen\-dieck category and let $(\AAA, \ttt_{\AAA})$ be a $k$-linear site. We say that a $k$-linear functor $u: (\AAA, \ttt_{\AAA}) \lra \ccc$ satisfies (LC), or is an \emph{LC morphism} provided that $u: (\AAA, \ttt_{\AAA}) \lra (\ccc, \ttt_{\ccc})$ satisfies (LC) where $\ttt_{\ccc}$ is the topology of jointly epimorphic sieves. Precisely, $R \in \ttt_{\ccc}(C)$ if and only if $\oplus_{(f: C_f \lra C) \in R} C_f \lra C$ is an epimorphism. The general Gabri\"el-Popescu theorem states that  for $\ttt_{\AAA} = u^{-1} \ttt_{\ccc}$, we have that $u$ is an LC morphism if and only if $\ttt_{\AAA}$ is a topology and $u$ gives rise to an equivalence $\ccc \lra \Sh(\AAA, \ttt_{\AAA})$ (see \cite{lowenGP}).

Consider $k$-linear Grothendieck categories $\ccc$ and $\ddd$.

\begin{proposition}\label{propwelldef}
Consider LC morphisms $u: (\AAA, \ttt_{\AAA}) \lra \ccc$, $u': (\AAA', \ttt_{\AAA'}) \lra \ccc$, $v: (\BBB, \ttt_{\BBB}) \lra \ddd$, $v': (\BBB', \ttt_{\BBB'}) \lra \ddd$. There exists an equivalence of categories
$$\Sh(\AAA \otimes \BBB, \ttt_{\AAA} \boxtimes \ttt_{\BBB}) \cong \Sh(\AAA' \otimes \BBB', \ttt_{\AAA'} \boxtimes \ttt_{\BBB'}).$$ 
\end{proposition}

\begin{proof}
Let $\CCC \subseteq \ccc$ be the full subcategory with $\Ob(\CCC) = \{ u(A) \,\, |\,\, A \in \AAA \} \cup \{ u'(A) \,\, |\,\, A' \in \AAA' \}$ and let $\DDD \subseteq \ddd$ be the full subcategory with $\Ob(\DDD) = \{ v(B) \,\, |\,\, B \in \BBB \} \cup \{ v'(B) \,\, |\,\, B' \in \BBB' \}$. Put $\ttt_{\CCC} = i^{-1} \ttt_{\ccc}$ for the inclusion $i: \CCC \lra \ccc$ and the canonical topology $\ttt_{\ccc}$ on $\ccc$ and put $\ttt_{\DDD} = j^{-1} \ttt_{\ddd}$ for the inclusion $j: \DDD \lra \ddd$ and the canonical topology $\ttt_{\ddd}$ on $\ddd$. It follows that the induced functors $\bar{u}: (\AAA, \ttt_{\AAA}) \lra (\CCC, \ttt_{\CCC})$, $\bar{u}': (\AAA', \ttt{\AAA'}) \lra (\CCC, \ttt_{\CCC})$, $\bar{v}: (\BBB, \ttt_{\BBB}) \lra (\DDD, \ttt_{\DDD})$, $\bar{v}': (\BBB', \ttt_{\BBB'}) \lra (\DDD, \ttt_{\DDD})$ are all LC morphisms. By Proposition \ref{propLC}, it follows that $\bar{u} \otimes \bar{v}: (\AAA \otimes \BBB, \ttt_{\AAA} \boxtimes \ttt_{\BBB}) \lra (\CCC \otimes \DDD, \ttt_{\CCC} \boxtimes \ttt_{\DDD})$ and $\bar{u}' \otimes \bar{v}': (\AAA' \otimes \BBB', \ttt_{\AAA'} \boxtimes \ttt_{\BBB'}) \lra (\CCC \otimes \DDD, \ttt_{\CCC} \boxtimes \ttt_{\DDD})$ are LC morphisms, and in particular we have equivalences of categories $\Sh(\AAA \otimes \BBB, \ttt_{\AAA} \boxtimes \ttt_{\BBB}) \cong \Sh(\CCC \otimes \DDD, \ttt_{\CCC} \boxtimes \ttt_{\DDD}) \cong \Sh(\AAA' \otimes \BBB', \ttt_{\AAA'} \boxtimes \ttt_{\BBB'})$.
\end{proof}
Thanks to Proposition \ref{propwelldef}, we can now make the following definition:

\begin{definition}\label{defthedef}
Consider Grothendieck categories $\ccc$ and $\ddd$. The \emph{tensor product} $\ccc \boxtimes \ddd$ is the following Grothendieck category, defined up to equivalence of categories: for arbitrary LC morphisms $u: (\AAA, \ttt_{\AAA}) \lra \ccc$ and $v: (\BBB, \ttt_{\BBB}) \lra \ddd$, we put
$$\ccc \boxtimes \ddd = \Sh(\AAA \otimes \BBB, \ttt_{\AAA} \boxtimes \ttt_{\BBB}).$$
\end{definition}

\subsection{Tensor product of $\Z$-algebras}\label{parzalg}
Recall that a $\Z$-algebra is a linear category $\AAA$ with $\Ob(\AAA) = \Z$. We further suppose that $\AAA$ is \emph{positively graded}, that is $\AAA(n,m) = 0$ for $n<m$. Following \cite{dedekenlowen}, we consider the sieves $\AAA(-,m)_{\geq n} \subseteq \AAA(-,m)$ for $n \geq m \in \Z$ with $$\AAA(l,m)_{\geq n} = \begin{cases} \AAA(l,m) & \text{if}\,\, l \geq n \\ 0 & \text{otherwise} \end{cases}$$
and we consider the \emph{tails localizing system} 
$$\LLL_{\mathrm{tails}} = \{ \AAA(-,m)_{\geq n} \,\, |\,\, n \geq m \}^{\mathrm{up}}$$
and the \emph{tails topology}
$$\ttt_{\mathrm{tails}} = \LLL_{\mathrm{tails}}^{\mathrm{upglue}}.$$

\begin{remark}
In many cases of interest, we have $\LLL_{\mathrm{tails}} = \ttt_{\mathrm{tails}}$. This is the case for a noetherian $\Z$-algebra or for a connected, finitely generated $\Z$-algebra in the sense of \cite{dedekenlowen}.
\end{remark}

For $\Z$-algebras $\AAA$ and $\BBB$, we define the \emph{diagonal} $\Z$-algebra $\CCC = (\AAA \otimes \BBB)_{\Delta}$ with
$$\CCC(n,m) = (\AAA \otimes \BBB)((n,n), (m,m)) = \AAA(n,m) \otimes \BBB(n,m).$$
There is a corresponding fully faithful functor
$$\Delta: \CCC \lra \AAA \otimes \BBB: n \longmapsto (n,n).$$
Let $\LLL_{\AAA}$, $\LLL_{\BBB}$, $\LLL_{\CCC}$ denote the tails localizing systems on $\AAA$, $\BBB$ and $\CCC$ respectively, and let $\ttt_{\AAA}$, $\ttt_{\BBB}$, $\ttt_{\CCC}$ denote the corresponding tails topologies. Further, consider the following cover system on $\AAA \otimes \BBB$:
$$\LLL_{\AAA \otimes \BBB} = \{ R \boxtimes S \,\, |\,\, R \in \LLL_{\AAA}, S \in \LLL_{\BBB} \}^{\mathrm{up}}$$

\begin{lemma}\label{lemloctentail}
The cover system $\LLL_{\AAA \otimes \BBB}$ is localizing and upclosed and $\ttt_{\AAA} \boxtimes \ttt_{\BBB} = \LLL_{\AAA \otimes \BBB}^{\mathrm{upglue}}$.
\end{lemma}

\begin{proposition}\label{propdelta}
The functor $\Delta: (\CCC, \ttt_{\CCC}) \lra (\AAA \otimes \BBB, \ttt_{\AAA} \boxtimes \ttt_{\BBB})$ is cocontinuous and we have $$\Delta^{-1}(\ttt_{\AAA} \boxtimes \ttt_{\BBB}) \subseteq \ttt_{\CCC}.$$
\end{proposition}

\begin{proof}
The second claim follows from the first one by \cite[Prop. 2.16]{lowenlin}
According to \cite[Lem. 2.12]{lowenlin}, it suffices to prove the statement for $\LLL_{\CCC}$ on $\CCC$ and $\LLL_{\AAA \otimes \BBB}$ on $\AAA \otimes \BBB$.  Hence, consider $m \in \CCC$ and $\Delta(m) = (m,m) \in \AAA \otimes \BBB$, and $R = \AAA(-,m)_{\geq n_1} \boxtimes \BBB(-,m)_{\geq n_2} \in \LLL_{\AAA \otimes \BBB}(m,m)$. For $n = \mathrm{max}(n_1, n_2)$, consider $S = \CCC(-,m)_{\geq n} \in \LLL_{\CCC}$. We have $S(l) = \AAA(l,m)_{\geq n} \otimes \BBB(l,m)_{\geq n} \subseteq \AAA(l,m)_{\geq n_1} \otimes \BBB(l,m)_{\geq n_2} = R(l,l)$ so $\Delta S \subseteq R$ as desired.
\end{proof}

In order to improve upon Proposition \ref{propdelta}, we look at generation of $\Z$-algebras in the sense of \cite{dedekenlowen}.

\begin{definition}\label{defgen}
\begin{enumerate}
\item A linear category $\AAA$ is \emph{generated} by subsets $X(A, A') \subseteq \AAA(A,A')$ if every element of $\AAA$ can be written as a linear sum of products of elements in $X$.
\item A $\Z$ algebra $\AAA$ is \emph{generated in certain degrees $D \subseteq \N$} if it is generated by $X$ with $X(n,m) = \varnothing$ unless $n - m \in D$.
\item A $\Z$-algebra $\AAA$ is \emph{finitely generated} if it is generated by $X$ such that for all $m$ the set $\cup_{d \in \N} X(m + d, m)$ is finite.
\item A $\Z$-algebra $\AAA$ is \emph{connected} if $\AAA(n,n) \cong k$ for all $n$.
\end{enumerate}
\end{definition}

We make the following observation:

\begin{proposition}
Consider $\Z$-algebras $\AAA$ and $\BBB$ and put $\CCC = (\AAA \otimes \BBB)_{\Delta}$.
\begin{enumerate}
\item If $\AAA$ is generated by $X_{\AAA}$ and $\BBB$ is generated by $X_{\BBB}$, then $\CCC$ is generated by $X_{\CCC}$ with $X_{\CCC}(n,m) = X_{\AAA}(n,m) \times X_{\BBB}(n,m)$.
\item If $\AAA$ and $\BBB$ are generated in degrees $D$ (resp. finitely generated, resp. connected), then so is $\CCC$.
\end{enumerate}
\end{proposition}

\begin{remark}
It was shown in \cite{dedekenlowen} that for a connected, finitely generated $\Z$-algebra $\AAA$, we have $\LLL_{\mathrm{tails}} = \ttt_{\mathrm{tails}}$.
\end{remark}

Our main result is the following:

\begin{theorem}\label{thmmaintails}
Consider $\Z$-algebras $\AAA$ and $\BBB$ which are generated in degree $1$. The functor $\Delta: (\CCC, \ttt_{\CCC}) \lra (\AAA \otimes \BBB, \ttt_{\AAA} \boxtimes \ttt_{\BBB})$ satisfies (LC). In particular, we have $\ttt_{\CCC} = \Delta^{-1}(\ttt_{\AAA} \boxtimes \ttt_{\BBB})$ and 
$$\Sh(\AAA, \ttt_{\AAA}) \boxtimes \Sh(\BBB, \ttt_{\BBB}) = \Sh(\AAA \otimes \BBB, \ttt_{\AAA} \boxtimes \ttt_{\BBB}) \cong \Sh(\CCC, \ttt_{\CCC}).$$
\end{theorem}

\begin{proof}
This follows from Proposition \ref{propdelta} and Lemmas \ref{lemdelta1} and \ref{lemdelta2}.
\end{proof}

\begin{lemma}\label{lemdelta1}
Suppose the $\Z$-algebras $\AAA$ and $\BBB$ are generated in degree $1$. The functor $\Delta: \CCC \lra (\AAA \otimes \BBB, \ttt_{\AAA} \boxtimes \ttt_{\BBB})$ satisfies (G).
\end{lemma}

\begin{proof}
Consider $(m_1, m_2) \in \AAA \otimes \BBB$. Suppose for instance that $m_2 \geq m_1$. Consider the cover $\AAA(-, m_1)_{\geq m_2} \boxtimes \BBB(-, m_2) \in \ttt_{\AAA} \boxtimes \ttt_{\BBB}(m_1, m_2)$. We claim that this cover is generated by the morphisms $x \otimes 1 \in \AAA(m_2, m_1) \otimes \BBB(m_2, m_2)$ from the diagonal element $(m_2, m_2)$ to $(m_1, m_2)$. Indeed, for an element $a \otimes b \in \AAA(l_1, m_1) \otimes \BBB(l_2, m_2)$ with $l_1 \geq m_2$, by the hypothesis on $\AAA$ we can write $a = \sum_{i = 1}^k a_i' a''_i$ for $a'_i \in \AAA(m_2, m_1)$ and $a''_i \in \AAA(l_1, m_2)$. Hence, $a \otimes b = \sum_{i = 1}^k(a'_i \otimes 1)(a''_i \otimes b)$ as desired.
\end{proof}

\begin{lemma}\label{lemdelta2}
Suppose the $\Z$-algebras $\AAA$ and $\BBB$ which are generated in degree $1$. We have $\ttt_{\CCC} \subseteq \Delta^{-1}(\ttt_{\AAA} \boxtimes \ttt_{\BBB})$.
\end{lemma}

\begin{proof}
Note that since $\Delta$ satisfies (G), (F) and (FF) with respect to $\ttt_{\AAA} \boxtimes \ttt_{\BBB}$ and $\Delta^{-1}(\ttt_{\AAA} \boxtimes \ttt_{\BBB})$, it follows by \cite[Thm. 2.13]{lowenlin} that the cover system $\Delta^{-1}(\ttt_{\AAA} \boxtimes \ttt_{\BBB})$ is a topology. Hence, to prove the desired inclusion, it suffices to show that $\LLL_{\CCC} \subseteq \Delta^{-1}(\ttt_{\AAA} \boxtimes \ttt_{\BBB})$. Consider $S = \CCC(-, m)_{\geq n} \in \LLL_{\CCC}(m)$. We are to show that $\Delta S \in \ttt_{\AAA} \boxtimes \ttt_{\BBB}(m,m)$. Now $\Delta S$ is generated by the morphisms in $\AAA(l,m) \otimes \BBB(l,m)$ for $l \geq n$. We claim that $\Delta S = \AAA(-,m)_{\geq n} \boxtimes \BBB(-,m)_{\geq n}$. To this end, we take an element $a \otimes b \in \AAA(l_1, m) \otimes \BBB(l_2, m)$ with $l_1, l_2 \geq n$. If for instance $l_2 \geq l_1$, by the hypothesis on $\BBB$, we can write $b = \sum_{i = 1}^k b'_i b''_i$ for $b'_i \in \BBB(l_1, m)$ and $b''_i \in \BBB(l_2, l_1)$ and hence $a \otimes b = \sum_{i = 1}^k (a \otimes b'_i)(1 \otimes b''_i) \in \Delta S$.
\end{proof}

\subsection{Quasicoherent sheaves on projective schemes}\label{parqcohproj}
Next we apply the results of \S \ref{parzalg} to graded algebras and schemes. A \emph{graded algebra} $A = (A_n)_{n \in \N}$ is an algebra $A = \oplus_{n \in \N} A_n$ with $1 \in A_0$ and multiplication determined by $A_n \otimes A_m \lra A_{n + m}$. Such an algebra has an associated $\Z$-algebra $\AAA(A)$ with $\AAA(A)(n, m) = A_{n-m}$. The algebra $A$ is generated in degrees $D \subseteq \N$ (resp. finitely generated, resp. connected) if and only if the associated $\Z$-algebra $\AAA(A)$ is. Now if $A$ is a finitely generated, connected graded algebra, the category $\mathsf{Gr}(A)$ of graded $A$-modules has a localizing subcategory $\mathsf{Tors}(A)$ of torsion modules, and one obtains the quotient category $\mathsf{Qgr}(A) = \mathsf{Gr}(A)/\mathsf{Tors}(A)$. By Serre's theorem, if $A$ is commutative with associated projective scheme $\mathrm{Proj}(A)$, we have $\Qch(\mathrm{Proj}(A)) = \mathsf{Qgr}(A)$. The category $\Qgr(A)$ has been generalized to certain classes of $\Z$-algebras in \cite{staffordvandenbergh}, \cite{vandenbergh2}, \cite{polishchuk} and in \cite{dedekenlowen}, the category $\Sh(\AAA, \ttt_{\mathrm{tails}})$ is introduced as a further generalization to arbitrary $\Z$-algebras. In particular, for a finitely generated connected graded algebra $A$, we have 
\begin{equation}\label{qgrsh}
\Qgr(A) \cong \Sh(\AAA(A), \ttt_{\mathrm{tails}})
\end{equation}

Next we turn to tensor products. For two graded algebras $A$ and $B$, the \emph{cartesian product} $A \times_{\mathrm{cart}} B$ is defined by $(A \times_{\mathrm{cart}} B)_n = A_n \otimes B_n$. We clearly have
\begin{equation}\label{eqcart}
\AAA(A \times_{\mathrm{cart}} B) = (\AAA(A) \otimes \AAA(B))_{\Delta}.
\end{equation}

\begin{theorem}\label{thmscheme}
\begin{enumerate}
\item For two graded algebras $A$ and $B$ which are connected and finitely generated in degree $1$, we have
$$\Qgr(A) \boxtimes \Qgr(B) = \Qgr(A \times_{\mathrm{cart}} B).$$
\item For two projective schemes $X$ and $Y$, we have
$$\Qch(X) \boxtimes \Qch(Y) = \Qch(X \times Y).$$
\end{enumerate} 
\end{theorem}

\begin{proof}
(1) Put $\AAA = \AAA(A)$, $\BBB = \AAA(B)$. According to \eqref{qgrsh} and Theorem \ref{thmmaintails}, we have
$$\Qgr(A) \boxtimes \Qgr(B) \cong \Sh(\AAA, \ttt_{\mathrm{tails}}) \boxtimes \Sh(\BBB, \ttt_{\mathrm{tails}}) \cong \Sh((\AAA \otimes \BBB)_{\Delta}, \ttt_{\mathrm{tails}})$$
and by \eqref{eqcart} and \eqref{qgrsh} , the category on the right hand side is isomorphic to $\Qgr(A \times_{\mathrm{cart}} B)$.
(2) It suffices to write $X \cong \mathrm{Proj}(A)$ and $Y \cong \mathrm{Proj}(B)$ for connected graded algebras generated in degree $1$.
\end{proof}

\begin{remark}
The formula $\Qch(X) \boxtimes \Qch(Y) = \Qch(X \times Y)$ should hold in greater generality, at least for schemes and suitable stacks. This will follow from the appropriate compatibility between tensor products and descent, and is work in progress.
\end{remark}

\section{Relation with other tensor products}
Our tensor product of Grothendieck categories is in close relation with two well-known tensor products of categories. In this section we analyse those relations.

The first one is the tensor product of locally presentable categories. It is well-known that every Grothendieck category is locally presentable. In \S \ref{parlocpres} we prove that taking our tensor product of two Grothendieck categories coincides with taking their tensor product as locally presentable categories. In particular, the class of locally $\alpha$-presentable Grothendieck categories for a fixed cardinal $\alpha$ is stable under our tensor product. This applies, for example, to the class of locally finitely presentable Grothendieck categories. This should be contrasted with the more restrictive class of locally coherent Grothendieck categories, which is not preserved, as is already clear from the ring case. 

The second one is Deligne's tensor product of small abelian categories. In \S \ref{pardeligne}, for small abelian categories $\aaa$ and $\bbb$ with associated Grothendieck categories $\Lex(\aaa)$ and $\Lex(\bbb)$ of left exact modules, based upon \cite{franco} the tensor product $\Lex(\aaa) \boxtimes \Lex(\bbb)$ is shown to be locally coherent precisely when the Deligne tensor product of $\aaa$ and $\bbb$ exists, and in this case the Deligne tensor product is given by the abelian category of finitely presented objects in $\Lex(\aaa) \boxtimes \Lex(\bbb)$.
Following a suggestion by Henning Krause, in \S \ref{paralphadel} we define an $\alpha$-version of the Deligne tensor product which is shown to underly any given tensor product of Grothendieck categories, as long as we choose $\alpha$ sufficiently large.

\subsection{Tensor product of locally presentable categories}\label{parlocpres}
Local presentability of categories is classically considered in a non-enriched context \cite{gabrielulmer}, for which enriched analogues exist \cite{kelly2}. In the case of $k$-linear categories, where enrichement is over $\Mod(k)$, the classical and the enriched notions of local presentability coincide. For the constructions considered in this section though, it is essential to work enriched over $\Mod(k)$. All categories and constructions in this section are to be understood in the $k$-linear sense.
 
Recall that a $k$-linear category $\ccc$ is \emph{locally presentable} if it is cocomplete and there exists a small regular cardinal $\alpha$ such that $\ccc$ has a set of strong generators consisting of \emph{$\alpha$-presented objects}, that is objects $G \in \ccc$ such that the $k$-linear functor $\ccc(G,-): \ccc \lra \Mod(k)$ preserves $\alpha$-filtered colimits. In this case the full subcategory $\ccc_{\alpha}$ of $\alpha$-presented objects is small, $\alpha$-cocomplete (i.e. closed under $\alpha$-small colimits) and it is obtained as the closure of the category of generators under $\alpha$-small colimits \cite{kelly2}. When we want to make the cardinal $\alpha$ explicit we will say $\mathcal{C}$ is \emph{locally $\alpha$-presentable}. 
Observe that this notion is a generalization to bigger cardinals of the notion of \emph{locally finitely presentable} $k$-linear category, which is obtained as the particular case with $\alpha = \aleph_0$. In that case we write $\fp(\ccc) = \ccc_{\aleph_0}$.

It is well known that Grothendieck categories are locally presentable (see for example \cite[Prop 3.4.16]{borceux}).

Consider $k$-linear categories $\aaa$, $\bbb$ and $\ccc$. We denote by $\Cont(\aaa, \bbb)$ (resp. by $\Cont_{\alpha}(\aaa, \bbb)$) the category of $k$-linear continuous (resp. $\alpha$-continuous) functors from $\aaa$ to $\bbb$, that is functors preserving all (existing) limits (resp. $\alpha$-small limits). We denote by $\Cont(\aaa, \bbb; \ccc)$ (resp. by $\Cont_{\alpha}(\aaa, \bbb; \ccc)$) the category of functors $\aaa \times \bbb \lra \ccc$ which are $k$-linear and continuous in each variable.

The categories $\Cocont(\aaa, \bbb)$, $\Cocont_{\alpha}(\aaa, \bbb)$, $\Cocont(\aaa, \bbb; \ccc)$ and $\Cocont_{\alpha}(\aaa, \bbb; \ccc)$ are defined similarly with limits replaced by colimits.

In the following theorem a tensor product of locally presentable categories is described.

\begin{theorem}\cite[Lem. 2.6, Rem. 2.7]{brandenburgchirvasitujohnsonfreyd}, \cite[\textsection2]{cavigliahorel}, \cite[Cor. 2.2.5]{chirvasitujohnsonfreyd}
Consider locally presentable $k$-linear categories $\aaa$ and $\bbb$.
	\begin{enumerate}
		\item The category $\Cocont(\aaa,\bbb)$ of $k$-linear cocontinuous functors is also a  locally presentable $k$-linear category.
		\item There exists a locally presentable $k$-linear category $\aaa \boxtimes_{\mathsf{LP}}\bbb$ such that for every cocomplete $k$-linear category $\ccc$ there is a natural equivalence of categories:
		$$\Cocont(\aaa \boxtimes_{\mathsf{LP}} \bbb, \ccc) \cong \Cocont(\aaa, \bbb;\ccc) \cong \Cocont(\aaa, \Cocont(\bbb,\ccc))$$
		\item In (2) we can take $\aaa \boxtimes_{\mathsf{LP}}\bbb = \Cont(\aaa\op, \bbb)$.
	\end{enumerate}
\end{theorem}

For small $\alpha$-cocomplete $k$-linear categories $\CCC$ and $\DDD$, we put $$\Lex_{\alpha}(\CCC) = \Cont_{\alpha}(\CCC\op, \Mod(k)) \subseteq \Mod(\CCC)$$ and $$\Lex_{\alpha}(\CCC, \DDD) = \Cont_{\alpha}(\CCC\op, \DDD\op; \Mod(k)) \subseteq \Mod(\CCC \otimes \DDD).$$ For $\alpha = \aleph_0$, we obtain the familiar categories $\Lex(\CCC) = \Lex_{\aleph_0}(\CCC)$ of left exact (that is, finite limit preserving) modules and $\Lex(\CCC,\DDD) = \Lex_{\aleph_0}(\CCC, \DDD)$ of modules that are left exact in both variables.

The category $\Lex_{\alpha}(\CCC)$ is locally $\alpha$-presentable, and we have $(\Lex_{\alpha}(\CCC))_{\alpha} \cong \CCC$. The category $\Lex_{\alpha}(\CCC)$ is the $\alpha$-free cocompletion of $\CCC$: every object in it can be written as an $\alpha$-filtered colimit of $\CCC$-objects, and according to \cite[Thm. 9.9]{kelly2}, for any cocomplete $k$-linear category $\ddd$ we have
\begin{equation}\label{alphafree}
\Cocont(\Lex_{\alpha}(\CCC), \ddd) = \Cocont_{\alpha}(\CCC, \ddd).
\end{equation}

Conversely, for a locally $\alpha$-presentable $k$-linear category $\ccc$, according to \cite[Thm. 7.2 + \textsection7.4]{kelly2}, we have 
\begin{equation}\label{lexalpha}
\ccc \cong \Lex_{\alpha}(\ccc_{\alpha}).
\end{equation}
One thus also obtains a natural \emph{$\alpha$-cocomplete tensor product} for small $\alpha$-cocomplete $k$-linear categories $\CCC$ and $\DDD$ \cite{kelly2}, \cite{kelly1}, given by
$$\CCC \otimes_{\alpha} \DDD = (\Lex_{\alpha}(\CCC) \boxtimes_{\mathsf{LP}} \Lex_{\alpha}(\DDD))_{\alpha}.$$
This $\alpha$-cocomplete tensor product satisfies the following universal property for every small $\alpha$-cocomplete $k$-linear category $\EEE$:
\begin{equation}\label{univalphatensor}
\Cocont_{\alpha}(\CCC \otimes_{\alpha} \DDD, \EEE) \cong \Cocont_{\alpha}(\CCC, \DDD; \EEE).
\end{equation}

For small finitely cocomplete categories $\CCC$ and $\DDD$, we denote $\CCC \otimes_{\mathsf{fp}} \DDD = \CCC \otimes_{\aleph_0} \DDD$.

The following alternative description of the tensor product of locally presentable categories is useful for our purpose. It appears for example in \cite{cavigliahorel}; we provide a proof for the convenience of the reader.

\begin{proposition} \label{prodlocpres}
	For locally $\alpha$-presentable $k$-linear categories $\ccc$ and $\ddd$, we have an equivalence
	$$\ccc \boxtimes_{\mathsf{LP}} \ddd \cong \Lex_{\alpha}(\ccc_{\alpha}, \ddd_{\alpha}).$$
\end{proposition}
\begin{proof}
	We have equivalences
	\begin{equation*}
	\begin{aligned}
	\ccc \boxtimes_{\mathsf{LP}} \ddd &= \Cont(\ccc\op,\ddd)\\
	&\cong \Cocont(\ccc, \ddd\op)\op\\
	&\cong \Cocont(\Lex_{\alpha}(\ccc_{\alpha}),\ddd\op)\op\\
	&\cong \Cocont_{\alpha}(\ccc_{\alpha}, \ddd\op)\op\\
	&\cong \Cont_{\alpha}(\ccc_{\alpha}\op, \Cont_{\alpha}(\ddd_{\alpha}\op, \Mod(k)))\\
	&\cong \Lex_{\alpha}(\ccc_{\alpha}, \ddd_{\alpha}),
	\end{aligned}
	\end{equation*}
	where we have used \eqref{lexalpha} in the third and fifth steps, \eqref{alphafree} in the fourth step, and the fact that limits are computed pointwise in $\Cont_{\alpha}(\ddd_{\alpha}\op, \Mod(k))$ in the last step.
\end{proof}

Next we turn our attention to ($k$-linear) Grothendieck categories.
The following result combines \cite[Thm 3.7]{lowenGP} and \cite[Thm 7.2 + \textsection7.4]{kelly2}, the latter being the enriched version of the analogous classical statement already formulated in \cite{gabrielulmer}: 
\begin{theorem}\label{thmlocpres}
	Let $\ccc$ be a locally $\alpha$-presentable Grothendieck category and consider the inclusion $u_{\ccc}: \ccc_{\alpha} \lra \ccc$. The canonical functor $\ccc \lra \Mod(\ccc_{\alpha}): C \longmapsto \ccc(u(-), C)$ factors through an equivalence of categories $$\ccc \lra \Sh(\ccc_{\alpha}, u_{\ccc}^{-1}\ttt_{\ccc}) = \Lex_{\alpha}(\ccc_{\alpha}).$$
\end{theorem}

We can now prove the main result of this section:

\begin{theorem}\label{maintwotensors}
For Grothendieck categories $\ccc$ and $\ddd$, we have an equivalence of categories
$$\ccc \boxtimes \ddd \cong \ccc \boxtimes_{\mathsf{LP}} \ddd.$$
\end{theorem}

\begin{proof}
Let $\alpha$ be a regular cardinal for which both $\ccc$ and $\ddd$ are locally $\alpha$-presentable.
By Theorem \ref{thmlocpres}, we have 
	$$\ccc \boxtimes \ddd = \Sh(\ccc_{\alpha} \otimes \ddd_{\alpha}, u_{\ccc}^{-1}\ttt_{\ccc} \boxtimes u_{\ddd}^{-1}\ttt_{\ddd}) = \Lex_{\alpha}(\ccc_{\alpha}, \ddd_{\alpha})$$
	since $\Lex_{\alpha}(\ccc_{\alpha}, \ddd_{\alpha})$ describes the intersection of the two one-sided sheaf categories following Theorem \ref{thmlocpres}. This finishes the proof by Proposition \ref{prodlocpres}.
\end{proof}

\begin{remark}
The way in which the tensor product $\boxtimes_{\mathsf{LP}}$ of locally presentable categories is defined through a universal property, makes it well-defined up to equivalence of categories. As an alternative to our independent approach to the tensor product of Grothendieck categories based upon functoriality, one can show in the spirit  of Proposition \ref{prodlocpres} that $\Sh(\AAA, \ttt_{\AAA}) \boxtimes_{\mathsf{LP}} \Sh(\BBB, \ttt_{\BBB}) = \Sh(\AAA \otimes \BBB, \ttt_{\AAA} \boxtimes \ttt_{\BBB})$.
\end{remark}

\begin{corollary}\label{cortensorclosed}
	The subclass of Grothendieck $k$-linear categories within the class of locally presentable $k$-linear categories is closed under the tensor product $\boxtimes_{\mathsf{LP}}$.
\end{corollary}

\subsection{Relation with Deligne's tensor product}\label{pardeligne}
In \cite{deligne}, Deligne defined a tensor product for abelian categories through a universal property. This tensor product is known to exist only under additional assumptions on the categories. Recall that a Grothendieck category $\ccc$ is \emph{locally coherent} if it is locally finitely presentable and $\fp(\ccc)$ is abelian. This defines a 1-1 correspondence between locally coherent Grothendieck categories on the one hand and small abelian categories on the other hand, the inverse being given by $\aaa \longmapsto \Lex(\aaa)$. For small abelian categories $\aaa$ and $\bbb$, according to \S \ref{parlocpres} we have
$$\Lex(\aaa) \boxtimes \Lex(\bbb) = \Lex(\aaa, \bbb).$$
Since the tensor product of coherent rings is not necessarily coherent (see for instance \cite[Ex. 21]{franco}), the tensor product of locally coherent Grothendieck categories is not necessarily locally coherent. We can complete \cite[Thm. 18]{franco} as follows:

\begin{theorem}\label{thmloccoh}
	For small abelian categories $\aaa$ and $\bbb$, the following are equivalent:
	\begin{enumerate}
		\item Deligne's tensor product of $\aaa$ and $\bbb$ exists;
		\item The tensor product $\aaa \otimes_{\mathrm{fp}} \bbb$ is abelian;
		\item The tensor product $\Lex(\aaa) \boxtimes \Lex(\bbb)$ is locally coherent.
	\end{enumerate}
	In this case, Deligne's tensor product equals $\aaa \otimes_{\mathrm{fp}} \bbb = \fp(\Lex(\aaa) \boxtimes \Lex(\bbb))$.
\end{theorem}

\subsection{The $\alpha$-Deligne tensor product}\label{paralphadel}

As suggested to us by Henning Krause, we define an $\alpha$-version of the Deligne tensor product for a cardinal $\alpha$ and we show that every tensor product of Grothendieck categories is accompanied by a parallel $\alpha$-Deligne tensor product of its categories of $\alpha$-presented objects for sufficiently large $\alpha$.

\begin{definition}
\begin{enumerate}
\item Let $\aaa$ and $\bbb$ be $\alpha$-cocomplete abelian categories. An \emph{$\alpha$-Deligne tensor product} of $\aaa$ and $\bbb$ is an $\alpha$-cocomplete abelian category $\aaa \bullet_{\alpha} \bbb$ with a functor $\aaa \otimes \bbb \lra \aaa \bullet_{\alpha} \bbb$ which is $\alpha$-cocontinuous in each variable and induces equivalences
$$\mathsf{Cocont}_{\alpha}(\aaa \bullet_{\alpha} \bbb, \ccc) \cong \mathsf{Cocont}_{\alpha}(\aaa, \bbb; \ccc)$$
for every $\alpha$-cocomplete abelian category $\ccc$.

\item Let $\aaa$ and $\bbb$ be abelian categories. If it exists, we define the \emph{modified $\alpha$-Deligne tensor product} to be $$\aaa \tilde{\bullet}_{\alpha} \bbb = \Lex(\aaa)_{\alpha} \bullet_{\alpha} \Lex(\bbb)_{\alpha}.$$
\end{enumerate}
\end{definition}

Note that for $\alpha = \aleph_0$, we have $\aaa \tilde{\bullet}_{\aleph_0} \bbb = \aaa \bullet_{\aleph_0} \bbb = \aaa \bullet \bbb$.

The following is proven along the lines of \cite[Prop. 6.1.13]{kashiwarashapira}, using the description of $\Lex_{\alpha}(\aaa) \cong \mathsf{Ind}_{\alpha}(\aaa)$ as ``ind completion'' in terms of $\alpha$-filtered colimits.

\begin{lemma}\label{lemablex}
Let $\aaa$ be a small $\alpha$-cocomplete abelian category. The category $\Lex_{\alpha}(\aaa)$ is abelian.
\end{lemma}

The following analogue of \cite[Lem. 17]{franco} is proven along the same lines, based upon Lemma \ref{lemablex}.

\begin{lemma}\label{lemcharlexdel}
Suppose for small $\alpha$-cocomplete abelian categories $\aaa$ and $\bbb$, the $\alpha$-Deligne tensor product $\aaa \bullet_{\alpha} \bbb$ exists. The category $\Lex_{\alpha}(\aaa \bullet_{\alpha} \bbb)$ is characterized by the following universal property for cocomplete abelian categories $\ccc$:
\begin{equation}
\Cocont(\Lex_{\alpha}(\aaa \bullet_{\alpha} \bbb), \ccc) = \Cocont_{\alpha}(\aaa, \bbb; \ccc).
\end{equation}
\end{lemma}

The following replacement of \cite[Thm. 18]{franco} is proven along the same lines. For $\alpha = \aleph_0$, note that the second part of condition (1) is automatically fulfilled.

\begin{theorem}\label{thmalphadel}
For $\alpha$-cocomplete abelian categories $\aaa$ and $\bbb$, the following are equivalent: 
\begin{enumerate}
\item The $\alpha$-Deligne tensor product $\aaa \bullet_{\alpha} \bbb$ exists and $\Lex_{\alpha}(\aaa \otimes_{\alpha} \bbb)$ is abelian;
\item The $\alpha$-cocomplete tensor product $\aaa \otimes_{\alpha} \bbb$ is abelian.
\end{enumerate}
In this case, we have $\aaa \bullet_{\alpha} \bbb = \aaa \otimes_{\alpha} \bbb$.
\end{theorem}

\begin{proof}
If $\aaa \otimes_{\alpha} \bbb$ is abelian, it obviously satisfies the universal property of $\aaa \bullet_{\alpha} \bbb$ and further, $\Lex_{\alpha}(\aaa \otimes_{\alpha} \bbb)$ is abelian by Lemma \ref{lemablex}. Conversely, suppose $\aaa \bullet_{\alpha} \bbb$ exists and $\Lex_{\alpha}(\aaa \otimes_{\alpha} \bbb)$ is abelian. The categories $\Lex_{\alpha}(\aaa \bullet_{\alpha} \bbb)$ and $\Lex_{\alpha}(\aaa \otimes_{\alpha} \bbb)$ have the categories $\aaa \bullet_{\alpha} \bbb$ and $\aaa \otimes_{\alpha} \bbb$ as respective categories of $\alpha$-presented objects, whence it suffices to show that $\Lex_{\alpha}(\aaa \otimes_{\alpha} \bbb)$, being cocomplete and abelian by assumption, has the universal property of Lemma \ref{lemcharlexdel}.
But this is clearly the case by \eqref{alphafree} \eqref{univalphatensor}.
\end{proof}

We recall the following:

\begin{proposition}\cite[Cor. 5.2]{krause5}\label{alphaabelian}
Let $\ccc$ be a Grothendieck category. There exists a cardinal $\alpha$ such that $\ccc$ is locally $\beta$-presentable and $\ccc_{\beta}$ is abelian for $\beta \geq \alpha$.
\end{proposition}

Whereas the tensor product of two Grothendieck categories cannot be related to the Deligne tensor product in general, it can always be related to an $\alpha$-Deligne tensor product in the following way:

\begin{proposition}
Let $\ccc$ and $\ddd$ be Grothendieck categories. There exists a cardinal $\alpha$ such that for $\beta \geq \alpha$ the $\beta$-Deligne tensor product $\ccc_{\beta} \bullet_{\beta} \ddd_{\beta}$ exists and we have $$(\ccc \boxtimes \ddd)_{\beta} \cong \ccc_{\beta} \bullet_{\beta} \ddd_{\beta}.$$
\end{proposition}

\begin{proof}
It suffices to note that by Proposition \ref{alphaabelian}, we can choose $\alpha$ such that for $\beta \geq \alpha$ the categories $\ccc$, $\ddd$ and $\ccc \boxtimes \ddd$ are locally $\beta$-presentable and $\ccc_{\beta}$, $\ddd_{\beta}$ and $(\ccc \boxtimes \ddd)_{\beta}$ are abelian. Hence, we have $\ccc_{\beta} \otimes_{\beta} \ddd_{\beta} \cong (\ccc \boxtimes \ddd)_{\beta}$ by Theorem \ref{maintwotensors}, and thus the desired isomorphism holds by Theorem \ref{thmalphadel}.
\end{proof}

As a special case, whereas two small abelian categories do not necessarily have a Deligne tensor product, they do have a modified $\alpha$-Deligne tensor product for sufficiently large $\alpha$:

\begin{corollary}
Let $\aaa$ and $\bbb$ be small abelian categories. There exists a cardinal $\alpha$ such that for $\beta \geq \alpha$ the modified $\beta$-Deligne tensor product $\aaa \tilde{\bullet}_{\beta} \bbb$ exists and we have $$(\Lex(\aaa, \bbb))_{\beta} \cong \aaa \tilde{\bullet}_{\beta} \bbb.$$
\end{corollary}

\subsection{Relation with the tensor product of toposes and future prospects}\label{parfuture}

In Theorem \ref{maintwotensors} we have shown that the tensor product of Grothendieck categories is a special instance of the tensor product of locally presentable linear categories, using special linear site presentations of the categories. This raises the natural question whether, if one takes the tensor product of locally presentable categories as starting point, there is a shorter route to the tensor product of Grothendieck categories than the one we followed. 

First one may note that in order to obtain an abstract tensor product of Grothen\-dieck categories, it suffices to prove Corollary \ref{maintwotensors} directly. As an anonymous referee suggested, one can prove along the lines of \cite[Cor. 15]{franco} that the tensor product of locally presentable categories preserves the Grothendieck property. However, this does not bring us any closer to the concrete expressions of the tensor product in terms of arbitrary representations in terms of linear sites, which is the main aim of the current paper.

Secondly one may note that, after proposing our concrete formula for the tensor product of Grothendieck categories using linear sites, it suffices to show that this formula satisfies the universal property of the tensor product of locally presentable linear categories in order to show at once that our formula leads to a good definition, and that Theorem \ref{maintwotensors} holds.
This approach indeed works, and is based upon the possibility to write down an analogous formula to \eqref{alphafree}, with regard to a linear site $(\AAA, \ttt)$. Precisely, for any cocomplete $k$-linear category $\ddd$ we have
\begin{equation}\label{pittskey}
\mathsf{Cocont}(\Sh(\AAA, \ttt), \ddd) = \mathsf{Cocont}_{\ttt}(\AAA, \ddd)
\end{equation}
where the right hand side denotes the category of linear functors $F: \AAA \lra \ddd$ whose induced colimit preserving functor $\hat{F}: \Mod(\AAA) \lra \ddd$ sends inclusions of covers $R \subseteq  \AAA(-,A)$ to isomorphisms in $\ddd$. 
Rather than spelling out the proof of the universal property for the tensor product in Definition \ref{defthedef}, we refer the reader to \cite{pitts} where the parallel reasoning is performed for toposes over $\Set$. 
The tensor product of Grothendieck categories which we have introduced can be seen as a linear counterpart to the \emph{product} of Grothendieck toposes which is described by Johnstone in \cite{johnstone}. 
In \cite{pitts}, Pitts shows that the product of Grothendieck toposes is a special instance of the ($\Set$-based) tensor product of locally presentable categories, using the universal property.

Unlike in the case of toposes, working over $\Mod(k)$ rather than over $\Set$, our tensor product does not describe a $2$-categorical product, but instead introduces a $2$-categorical monoidal structure on linear toposes. Further, we should note that the establishment of the correct formula for the tensor product does not automatically yield the tangible functoriality properties for linear sites which we have proven. With our motivation coming from non-commutative geometry, it is precisely the flexibility in choosing appropriate sites, and the possibility to view certain functors of geometric origin as induced by natural morphisms of sites, which is of greatest interest to us. 

The notion of LC morphism which we prove in Proposition \ref{propLC} to be stable under the tensor product, is more restrictive than a morphism inducing an equivalence on the level of sheaf categories, and so this result cannot be deduced a posteriori from the existence of the tensor product satisfying the universal property. In fact, the class of LC morphisms opens up the interesting possibility to describe the ``category of Grothendieck categories'' up to equivalence as a 2-category of fractions, obtained from the category of linear sites by inverting LC morphisms. This fact, and its implications for the tensor product, will be elaborated further in \cite{juliathesis}.

On the other hand, a combination of Pitts' approach and our description of the tensor product in terms of localizing Serre subcategories leads to a natural tensor product for 
well-generated algebraic triangulated categories, which stand in relation to derived categories of differential graded algebras like Grothendieck categories stand in relation to module categories according to \cite{porta}. To make this idea precise, one takes To\"en's inner hom between dg categories as starting point, and between (homologically) cocomplete (with respect to arbitrary set indexed coproducts) dg categories one considers its restriction $\RHom_c(\aaa, \bbb)$ to bimodules inducing cocontinuous functors on the level of homology, inspired upon \cite[\S 7]{toen}. The \emph{cocomplete tensor product} between cocomplete dg categories $\aaa$ and $\bbb$ is by definition, if it exists, the unique cocomplete dg category $\aaa \boxtimes_c \bbb$ satisfying the following universal property with respect to cocomplete dg categories $\ccc$:
\begin{equation}\label{eqbox}
\RHom_c(\aaa \boxtimes_c \bbb, \ccc) \cong \RHom_c(\aaa, \RHom_c(\bbb, \ccc)).
\end{equation}
In the dg world, dg topologies are not quite the right tool in orther to perform localization on the derived level. We define a \emph{dg site} as a small dg category $\AAA$ along with a localizing thick subcategory $\www \subseteq D(\AAA)$ of the derived category. If by $\underline{D}(\AAA)$ we denote the dg derived category, then the dg quotient $\underline{D}(\AAA)$ by $\www$ can be characterized by the following replacement of \eqref{pittskey}
\begin{equation}
\RHom_c(\underline{D}(\AAA)/\www, \ccc) \cong \RHom_{\www}(\AAA, \ccc)
\end{equation}
where the right hand side denotes the subcategory of $\RHom(\AAA, \ccc)$ consisting of the bimodules for which the induced cocontinuous functor $D(\AAA) \lra H^0(\ccc)$ sends $\www$ to zero. With a definition inspired upon \S \ref{partensorlocsub}, one can define the tensor product of well-generated triangulated categories and show that it satisfies the universal property \eqref{eqbox}. The development of this approach, as well as its precise relation to the tensor product of Grothendieck categories, in particular under suitable flatness hypothesis like the one from \cite{lowenvandenberghab}, are work in progress and will appear in \cite{juliathesis}. Further, the definition should also be related to the tensor product of locally presentable infinity categories \cite[\S 4.1]{lurie}.

\def\cprime{$'$} \def\cprime{$'$}
\providecommand{\bysame}{\leavevmode\hbox to3em{\hrulefill}\thinspace}
\providecommand{\MR}{\relax\ifhmode\unskip\space\fi MR }
\providecommand{\MRhref}[2]{%
  \href{http://www.ams.org/mathscinet-getitem?mr=#1}{#2}
}
\providecommand{\href}[2]{#2}


\begin{thebibliography}{10}

\bibitem{artingrothendieckverdier1}
\emph{Th\'eorie des topos et cohomologie \'etale des sch\'emas. {T}ome 1:
  {T}h\'eorie des topos}, Springer-Verlag, Berlin, 1972, S\'eminaire de
  G\'eom\'etrie Alg\'ebrique du Bois-Marie 1963--1964 (SGA 4), Dirig\'e par M.
  Artin, A. Grothendieck, et J. L. Verdier. Avec la collaboration de N.
  Bourbaki, P. Deligne et B. Saint-Donat, Lecture Notes in Mathematics, Vol.
  269. \MR{MR0354652 (50 \#7130)}

\bibitem{adamekrosicky}
Ji{\v{r}}{\'{\i}} Ad{\'a}mek and Ji{\v{r}}{\'{\i}} Rosick{\'y}, \emph{Locally
  presentable and accessible categories}, London Mathematical Society Lecture
  Note Series, vol. 189, Cambridge University Press, Cambridge, 1994.
  \MR{MR1294136 (95j:18001)}

\bibitem{artintatevandenbergh}
M.~Artin, J.~Tate, and M.~Van~den Bergh, \emph{Some algebras associated to
  automorphisms of elliptic curves}, The {G}rothendieck {F}estschrift, {V}ol.\
  {I}, Progr. Math., vol.~86, Birkh\"auser Boston, Boston, MA, 1990,
  pp.~33--85. \MR{MR1086882 (92e:14002)}

\bibitem{artinzhang2}
M.~Artin and J.~J. Zhang, \emph{Noncommutative projective schemes}, Adv. Math.
  \textbf{109} (1994), no.~2, 228--287. \MR{MR1304753 (96a:14004)}

\bibitem{bondalpolishchuk}
A.~I. Bondal and A.~E. Polishchuk, \emph{Homological properties of associative
  algebras: the method of helices}, Izv. Ross. Akad. Nauk Ser. Mat. \textbf{57}
  (1993), no.~2, 3--50. \MR{MR1230966 (94m:16011)}

\bibitem{borceux}
F.~Borceux, \emph{Handbook of categorical algebra. 3}, Encyclopedia of
  Mathematics and its Applications, vol.~52, Cambridge University Press,
  Cambridge, 1994, Categories of sheaves. \MR{MR1315049 (96g:18001c)}

\bibitem{brandenburgchirvasitujohnsonfreyd}
Martin Brandenburg, Alexandru Chirvasitu, and Theo Johnson-Freyd,
  \emph{Reflexivity and dualizability in categorified linear algebra}, Theory
  Appl. Categ. \textbf{30} (2015), 808--835. \MR{3361309}

\bibitem{verschoren1}
J.~L. Bueso, P.~Jara, and A.~Verschoren, \emph{Compatibility, stability, and
  sheaves}, Monographs and Textbooks in Pure and Applied Mathematics, vol. 185,
  Marcel Dekker Inc., New York, 1995. \MR{MR1300631 (95i:16029)}

\bibitem{cavigliahorel}
G.~{Caviglia} and G.~{Horel}, \emph{{Rigidification of higher categorical
  structures}}, {\tt arXiv:1511.01119v2 [math.AT]}.

\bibitem{chirvasitujohnsonfreyd}
Alex Chirvasitu and Theo Johnson-Freyd, \emph{The fundamental pro-groupoid of
  an affine 2-scheme}, Appl. Categ. Structures \textbf{21} (2013), no.~5,
  469--522. \MR{3097055}

\bibitem{dedekenlowen}
O.~De~Deken and W.~Lowen, \emph{Abelian and derived deformations in the
  presence of {$\Bbb Z$}-generating geometric helices}, J. Noncommut. Geom.
  \textbf{5} (2011), no.~4, 477--505. \MR{2838522}

\bibitem{deligne}
P.~Deligne, \emph{La conjecture de {W}eil. {II}}, Inst. Hautes \'Etudes Sci.
  Publ. Math. (1980), no.~52, 137--252. \MR{MR601520 (83c:14017)}

\bibitem{deligne1}
\bysame, \emph{Cat\'egories tannakiennes}, The {G}rothendieck {F}estschrift,
  {V}ol.\ {II}, Progr. Math., vol.~87, Birkh\"auser Boston, Boston, MA, 1990,
  pp.~111--195. \MR{1106898}

\bibitem{eisenbudcommag}
D.~Eisenbud, \emph{Commutative algebra}, Graduate Texts in Mathematics, vol.
  150, Springer-Verlag, New York, 1995, With a view toward algebraic geometry.

\bibitem{gabrielulmer}
Peter Gabriel and Friedrich Ulmer, \emph{Lokal pr\"asentierbare {K}ategorien},
  Lecture Notes in Mathematics, Vol. 221, Springer-Verlag, Berlin-New York,
  1971. \MR{0327863}
  
\bibitem{johnstone}
P.~T. Johnstone, \emph{Sketches of an elephant: a topos theory compendium Vol. 1}, Oxford Logic Guides, vol 43, The Clarendon Press, Oxford university Press, New York, 2002. \MR{1953060}

\bibitem{kaledinlowen}
D.~Kaledin and W.~Lowen, \emph{Cohomology of exact categories and
  (non-)additive sheaves}, Adv. Math. \textbf{272} (2015), 652--698.
  \MR{3303245}

\bibitem{kashiwarashapira}
Masaki Kashiwara and Pierre Schapira, \emph{Categories and sheaves},
  Grundlehren der Mathematischen Wissenschaften [Fundamental Principles of
  Mathematical Sciences], vol. 332, Springer-Verlag, Berlin, 2006. \MR{2182076}

\bibitem{keller6}
B.~Keller, \emph{Derived invariance of higher structures on the {H}ochschild
  complex}, preprint
  \texttt{http://www.math.jussieu.fr/\~{}keller/publ/dih.dvi}.

\bibitem{kelly2}
G.~M. Kelly, \emph{Structures defined by finite limits in the enriched context.
  {I}}, Cahiers Topologie G\'eom. Diff\'erentielle \textbf{23} (1982), no.~1,
  3--42, Third Colloquium on Categories, Part VI (Amiens, 1980). \MR{648793}

\bibitem{kelly1}
\bysame, \emph{Basic concepts of enriched category theory}, Repr. Theory Appl.
  Categ. (2005), no.~10, vi+137, Reprint of the 1982 original [Cambridge Univ.
  Press, Cambridge; MR0651714]. \MR{2177301}

\bibitem{krause5}
Henning Krause, \emph{Deriving {A}uslander's formula}, Doc. Math. \textbf{20}
  (2015), 669--688. \MR{3398723}

\bibitem{franco}
I.~L{\'o}pez~Franco, \emph{Tensor products of finitely cocomplete and abelian
  categories}, J. Algebra \textbf{396} (2013), 207--219. \MR{3108080}

\bibitem{lowenlin}
W.~Lowen, \emph{Linearized topologies and deformation theory}, Topology and its
  Applications, doi:10.2016/j.topol.2015.12.019.

\bibitem{lowenGP}
\bysame, \emph{A generalization of the {G}abriel-{P}opescu theorem}, J. Pure
  Appl. Algebra \textbf{190} (2004), no.~1-3, 197--211. \MR{MR2043328}

\bibitem{lowenvandenberghhoch}
W.~Lowen and M.~Van~den Bergh, \emph{Hochschild cohomology of abelian
  categories and ringed spaces}, Advances in Math. \textbf{198} (2005), no.~1,
  172--221.

\bibitem{lowenvandenberghab}
\bysame, \emph{Deformation theory of abelian categories}, Trans. Amer. Math.
  Soc. \textbf{358} (2006), no.~12, 5441--5483.

\bibitem{lurie}
J.~Lurie, \emph{{Derived Algebraic Geometry II: Noncommutative Algebra}}, {\tt
  arXiv:math/0702299v5 [math.CT]}.
  
\bibitem{pitts}
A.~Pitts, \emph{On product and change of base for toposes}, Cahiers Topologie G{\'e}om. Diff{\'e}rentielle Cat{\'e}g. \textbf{26} (1985), no.~1, 43--61. \MR{794416}


\bibitem{polishchuk}
A.~Polishchuk, \emph{Noncommutative proj and coherent algebras}, Math. Res.
  Lett. \textbf{12} (2005), no.~1, 63--74. \MR{MR2122731 (2005k:14003)}

\bibitem{popescugabriel}
N.~Popescu and P.~Gabriel, \emph{Caract\'erisation des cat\'egories
  ab\'eliennes avec g\'en\'erateurs et limites inductives exactes}, C. R. Acad.
  Sci. Paris \textbf{258} (1964), 4188--4190. \MR{MR0166241 (29 \#3518)}

\bibitem{porta}
M.~Porta, \emph{The {P}opescu-{G}abriel theorem for triangulated categories},
  Adv. Math. \textbf{225} (2010), no.~3, 1669--1715. \MR{2673743}
  
  
  \bibitem{juliathesis}
  J.~Ramos Gonz\'alez, \emph{On tensor products of large categories}, PhD thesis in preparation, University of Antwerp.
  
  

\bibitem{staffordvandenbergh}
J.~T. Stafford and M.~van~den Bergh, \emph{Noncommutative curves and
  noncommutative surfaces}, Bull. Amer. Math. Soc. (N.S.) \textbf{38} (2001),
  no.~2, 171--216 (electronic).
  \MR{MR1816070 (2002d:16036)}
  
  
  \bibitem{toen}
  B.~To{\"e}n, \emph{The homotopy theory of {$dg$}-categories and derived {M}orita
              theory}, Invent. Math. \textbf{167} (2007) no.~3, 615--667.
  
   
  

\bibitem{vandenbergh2}
M.~Van~den Bergh, \emph{Noncommutative quadrics}, Int. Math. Res. Not. IMRN
  (2011), no.~17, 3983--4026. \MR{2836401}

\end{thebibliography}
\end{document}